\documentclass[12pt, a4paper]{amsart}
\usepackage{tikz}
\usepackage{amsmath,amsthm,amsfonts,amssymb,latexsym,mathrsfs,color,extarrows}
\usepackage{color, graphicx, comment}

\allowdisplaybreaks

\newtheorem{Theorem}{Theorem}
\newtheorem{Corollary}[Theorem]{Corollary}

\newtheorem{Conjecture}[Theorem]{Conjecture}

\newtheorem{Lemma}[Theorem]{Lemma}
\newtheorem{Definition}[Theorem]{Definition}

\allowdisplaybreaks

\DeclareMathOperator{\NN}{\mathbb{N}}

\DeclareMathOperator{\QQ}{\mathbb{Q}}

\DeclareMathOperator{\HF}{\mathcal{H}}
\DeclareMathOperator{\RHS}{\mathrm{RHS}}

\def\cFrac#1#2{%
  \begin{array}{@{}c@{}}\multicolumn{1}{c|}{#1}\\%
	\hline\multicolumn{1}{|c}{#2}\end{array}\;}

\title[Convolution powers of Narayana polynomials]{Hankel determinants for convolution powers of Narayana polynomials} 
\date{December 9, 2025}

\author{Guo-Niu Han}
\address{I.R.M.A., UMR 7501, Universit\'e de Strasbourg
et CNRS, 7 rue Ren\'e Descartes, F-67084 Strasbourg, France}
\email{guoniu.han@unistra.fr}



\begin{document}
\begin{abstract} 
We prove and generalize a conjecture of Johann Cigler on the Hankel determinants of convolution powers of Narayana polynomials. Our method follows a “guess-and-prove” strategy, relying on established techniques involving Hankel continued fractions. While the final forms of our theorems are given by simple closed expressions, the proofs require us to formulate and manage extremely large and intricate explicit expressions at intermediate stages. Most of the technically involved and lengthy formal verifications are carried out using a symbolic computation program, whose code is available on the author's personal webpage for independent verification.
We emphasize that our program delivers rigorous symbolic proofs, rather than merely verifying the initial terms.
\end{abstract}

\maketitle

\section{Introduction} 

The purpose of this paper is to establish and extend a conjecture of Johann Cigler concerning the Hankel determinants of convolution powers of Narayana polynomials \cite{Cigler2018}.  
Let ${\color{blue}\gamma}_{\color{purple}n} = \frac{1}{{\color{purple}n}+1} \binom{2{\color{purple}n}}{{\color{purple}n}}$ denote the ${\color{purple}n}$-th {\it Catalan number}. 
Hankel determinants formed from the Catalan sequence have been investigated in 
\cite{Aigner1999_Catalanlike, DeSainte1986Viennot, Krattenthaler2005, Tamm2001_Aspects, Cigler2011, Chang2013HuZhang_Hdet, Han2025Pedon}. 
In 2002, Cvetkovi\'c, Rajkovi\'c, and Ivkovi\'c \cite{Cvetkovic2002RI_CatalanHDet} determined the Hankel determinants of the sequence whose entries are sums of consecutive Catalan numbers:
\begin{equation}\label{sumCC}\det\left({\color{blue}\gamma}_{{\color{violet}i}+{\color{blue}j}}+{\color{blue}\gamma}_{{\color{violet}i}+{\color{blue}j}+1}\right)_{{\color{violet}i},{\color{blue}j}=0}^{{\color{purple}n}-1}
={\color{blue}F}_{2{\color{purple}n}+1}, 
\end{equation}
where ${\color{blue}F}_{\color{purple}n}$ denotes the ${\color{purple}n}$-th {\it Fibonacci number}.
%
Since then, this identity has been extended in various directions 
\cite{Chamberland2007French_GenCatalan, Krattenthaler2010_genCatalan, Egecioglu2009_CatalanHdet, Cigler2021Kratt_linear, Krattenthaler2021_linearII, Wang2018Xin_Convolution, Cigler2023_Experimental, Rajkovic2007PB_HCatalan}.
Motivated by \eqref{sumCC}, 
Cigler \cite{Cigler2018} studied the Hankel determinants associated with convolution powers of Narayana polynomials. Recall that
the {\it Narayana polynomials} are given by
\begin{equation}\label{def:Narayana}{\color{blue}\gamma}_{\color{purple}n}({\color{cyan}t}) = \sum_{{\color{olive}k}=0}^{\color{purple}n} \binom{{\color{purple}n}}{{\color{olive}k}} \binom{{\color{purple}n}-1}{{\color{olive}k}} \frac{1}{{\color{olive}k}+1} {\color{cyan}t}^{\color{olive}k},
\end{equation}
which specialize to the Catalan numbers when ${\color{cyan}t}=1$.  The initial terms of the sequence $({\color{blue}\gamma}_{\color{purple}n}({\color{cyan}t}))_{{\color{purple}n}\geq 0}$ are
$$
1,1,1+{\color{cyan}t}, 1+3{\color{cyan}t}+{\color{cyan}t}^2, 1+6{\color{cyan}t}+6{\color{cyan}t}^2+{\color{cyan}t}^3, 1+10{\color{cyan}t}+20{\color{cyan}t}^2+10{\color{cyan}t}^3+{\color{cyan}t}^4, \ldots
$$
The generating function of the Narayana polynomials
\begin{equation*}
{\color{blue}\gamma}({\color{cyan}t},{\color{magenta}q}) = \sum_{{\color{purple}n}\geq 0} {\color{blue}\gamma}_{\color{purple}n}({\color{cyan}t}) {\color{magenta}q}^{\color{purple}n}
\end{equation*}
satisfies the quadratic relation
\begin{equation}\label{eq:Narayana}
-1 + (1 - {\color{magenta}q} + {\color{cyan}t}{\color{magenta}q})\,{\color{blue}\gamma}({\color{cyan}t},{\color{magenta}q}) - {\color{cyan}t}{\color{magenta}q}\,{\color{blue}\gamma}({\color{cyan}t},{\color{magenta}q})^2 = 0.
\end{equation}
The {\it convolution powers of the Narayana polynomials} ${\color{blue}\gamma}^{({\color{olive}\tau})}_{\color{purple}n}({\color{cyan}t})$ are defined via the following generating functions, depending on the parity of ${\color{olive}\tau}$:
\begin{align*}
\sum {\color{blue}\gamma}^{(2{\color{olive}\tau})}_{\color{purple}n}({\color{cyan}t}) {\color{magenta}q}^{\color{purple}n} &= {\color{red}G}({\color{cyan}t},{\color{magenta}q})^{\color{olive}\tau},\\
\sum {\color{blue}\gamma}^{(2{\color{olive}\tau}+1)}_{\color{purple}n}({\color{cyan}t}) {\color{magenta}q}^{\color{purple}n} &= {\color{blue}\gamma}({\color{cyan}t},{\color{magenta}q})\,{\color{red}G}({\color{cyan}t},{\color{magenta}q})^{\color{olive}\tau},
\end{align*}
where ${\color{red}G}({\color{cyan}t},{\color{magenta}q})=\frac{{\color{blue}\gamma}({\color{cyan}t},{\color{magenta}q})-1}{{\color{magenta}q}}$ denotes the generating function of the shifted Narayana polynomial sequence.
Observe that
$$
{\color{blue}\gamma}^{({\color{olive}\tau})}_{\color{purple}n}(1) = \frac{{\color{olive}\tau}}{2{\color{purple}n}+{\color{olive}\tau}} \binom{2{\color{purple}n}+{\color{olive}\tau}}{{\color{purple}n}}
$$
gives the ${\color{olive}\tau}$-fold convolution power of the Catalan numbers.
In \cite{Cigler2018}, Cigler investigated the Hankel determinants associated with the convolution powers of Narayana polynomials ${\color{blue}\gamma}^{({\color{olive}\tau})}_{\color{purple}n}({\color{cyan}t})$ and derived explicit expressions for the cases ${\color{olive}\tau} = 3$ and ${\color{olive}\tau} = 4$. 
He also proposed a conjecture for ${\color{olive}\tau} = 6$. 
We restate these results below. Let
\begin{equation}\label{def:Delta}
{\color{purple}\Delta}^{({\color{red}m})}_{{\color{purple}n}} =  \det \left({{\color{blue}\gamma}}^{({\color{red}m})}_{{\color{purple}n}} ({\color{cyan}t})  \right)_{{\color{violet}i},{\color{blue}j}=0}^{{\color{violet}N}-1}.
\end{equation}

\begin{Theorem}[\cite{Cigler2018}, Theorem 5.2]\label{th:cigler52}
	We have
\begin{align}
{\color{purple}\Delta}^{(3)}_{{\color{violet}i}+{\color{blue}j}} 
&={\color{cyan}t}^{\binom{{\color{violet}N}}{2}} \sum_{{\color{olive}k}=0}^{\lfloor {\color{violet}N}/2 \rfloor} (-1)^{\color{olive}k} \binom{{\color{violet}N}-{\color{olive}k}}{{\color{olive}k}} {\color{cyan}t}^{-{\color{olive}k}}.\label{th:Cigler5.2}
\end{align}
\end{Theorem}

\begin{Theorem}[\cite{Cigler2018}, Theorem 7.4]\label{th:cigler74}
	We have
\begin{align}
{\color{purple}\Delta}^{(4)}_{{\color{violet}i}+{\color{blue}j}} 
&=\begin{cases}
	(-1)^{\color{purple}n} {\color{cyan}t}^{2{\color{purple}n}({\color{purple}n}-1)} [{\color{purple}n}+1]_{{\color{cyan}t}^2} , & \text{if ${\color{violet}N}=2{\color{purple}n}$}, \\
	(-1)^{\color{purple}n} {\color{cyan}t}^{2{\color{purple}n}^2} [{\color{purple}n}+1]_{{\color{cyan}t}^2}, & \text{if ${\color{violet}N}=2{\color{purple}n}+1$}, 
\end{cases}\label{th:Cigler7.4}
\end{align}
where $[{\color{purple}n}]_{\color{magenta}q}$ is the standard notation for ${\color{magenta}q}$-number:
$$ [{\color{purple}n}]_{\color{magenta}q} = 1+{\color{magenta}q} + {\color{magenta}q}^2 + \cdots + {\color{magenta}q}^{{\color{purple}n}-1}.  $$

\end{Theorem}

\begin{Conjecture}[\cite{Cigler2018}, Conjecture 7.6]\label{conj:cigler76}
	We have
\begin{align}
{\color{purple}\Delta}^{(6)}_{{\color{violet}i}+{\color{blue}j}} 
&=\begin{cases}
	(-1)^{\color{purple}n} {\color{cyan}t}^{9{\color{purple}n}({\color{purple}n}-1)/2} [{\color{purple}n}+1]^2_{{\color{cyan}t}^3} , & \text{if ${\color{violet}N}=3{\color{purple}n}$}, \\
	(-1)^{\color{purple}n} {\color{cyan}t}^{3{\color{purple}n}(3{\color{purple}n}-1)/2} [{\color{purple}n}+1]^2_{{\color{cyan}t}^3} , & \text{if ${\color{violet}N}=3{\color{purple}n}+1$}, \\
	(-1)^{{\color{purple}n}+1} 3 {\color{cyan}t}^{3{\color{purple}n}(3{\color{purple}n}+1)/2} [3]_{\color{cyan}t}  {\color{teal}r}_{\color{purple}n}({\color{cyan}t}), & \text{if ${\color{violet}N}=3{\color{purple}n}+2$},
\end{cases}\label{conj:Cigler}
\end{align}
where
\begin{align*}
		{\color{teal}r}_{\color{purple}n}({\color{cyan}t}) &= 1+ 3{\color{cyan}t}^3 + 6{\color{cyan}t}^6 + \cdots + \binom{{\color{purple}n}+1}{2} {\color{cyan}t}^{3({\color{purple}n}-1)} \\
		& + \binom{{\color{purple}n}+2}{2} {\color{cyan}t}^{3{\color{purple}n}} + \binom{{\color{purple}n}+1}{2} {\color{cyan}t}^{3({\color{purple}n}+1)} + \cdots + {\color{cyan}t}^{6{\color{purple}n}}.
\end{align*}

\end{Conjecture}
For further references on these Hankel determinants, we refer the reader to
\cite{Wang2018Xin_Convolution, Cigler2023_Experimental, Cigler2024ConvCatalan,
Egecioglu2007RR_AlmostHDet, Gessel2006Xin}.

\medskip

In this paper, we establish generalizations of Theorem \ref{th:cigler74} and Conjecture \ref{conj:cigler76}.
Rather than working with ${\color{red}G}({\color{cyan}t},{\color{magenta}q})^{{\color{red}m}}$,
we focus on the sequence generated by $({\color{blue}\gamma}({\color{cyan}t},{\color{magenta}q})-1)^{{\color{red}m}}$,
and derive closed-form expressions for the Hankel determinants of this sequence
as well as for its first, second, and third shifted versions:
$$
{({\color{blue}\gamma}({\color{cyan}t},{\color{magenta}q})-1)^{{\color{red}m}}}, 
\frac{({\color{blue}\gamma}({\color{cyan}t},{\color{magenta}q})-1)^{{\color{red}m}}}{{\color{magenta}q}},
\frac{({\color{blue}\gamma}({\color{cyan}t},{\color{magenta}q})-1)^{{\color{red}m}}}{{\color{magenta}q}^2},
\frac{({\color{blue}\gamma}({\color{cyan}t},{\color{magenta}q})-1)^{{\color{red}m}}}{{\color{magenta}q}^3}.
$$
It is straightforward to see that the second shifted sequence coincides with ${\color{blue}\gamma}^{(4)}_{\color{purple}n}({\color{cyan}t})$ when ${\color{red}m}=2$,
and that the third shifted sequence coincides with ${\color{blue}\gamma}^{(6)}_{\color{purple}n}({\color{cyan}t})$ when ${\color{red}m}=3$.
We also set ${\color{blue}\gamma}^{(2{\color{red}m})}_{\color{violet}i}({\color{cyan}t})=0$ for ${\color{violet}i}<0$.
Our main theorems are presented below, where we write $ {\color{cyan}\xi}_1 = (-1)^{{\color{purple}n}{\color{red}m}({\color{red}m}-1)/2}$ for short.

\begin{Theorem}\label{th:main:Det0}
For ${\color{red}m}\geq 1$ we have
\begin{align*}
{\color{purple}\Delta}^{(2{\color{red}m})}_{{\color{violet}i}+{\color{blue}j}-{\color{red}m}} 
&=\begin{cases}
	1, & \text{if ${\color{violet}N}=0$},\\
	-{\color{cyan}\xi}_1 {\color{cyan}t}^{{\color{red}m}({\color{purple}n}-1)({\color{red}m}{\color{purple}n}-2)/2} [{\color{purple}n}-1]_{{\color{cyan}t}^{\color{red}m}} , & \text{if ${\color{violet}N}={\color{red}m}{\color{purple}n}$}, \\
	(-1)^{{\color{red}m}}{\color{cyan}\xi}_1 {\color{cyan}t}^{{\color{red}m}^2 {\color{purple}n}({\color{purple}n}-1)/2} [{\color{purple}n}]_{{\color{cyan}t}^{\color{red}m}} , & \text{if ${\color{violet}N}={\color{red}m}{\color{purple}n}+1$}, \\
	0, & \text{otherwise}.
\end{cases}
\end{align*}
\end{Theorem}

\begin{Theorem}\label{th:main:Det1}
For  ${\color{red}m}\geq 1$ we have
\begin{align*}
{\color{purple}\Delta}^{(2{\color{red}m})}_{{\color{violet}i}+{\color{blue}j}+1-{\color{red}m}} 
&=\begin{cases}
	{\color{cyan}\xi}_1 {\color{cyan}t}^{{\color{red}m}^2{\color{purple}n}({\color{purple}n}-1)/2} , & \text{if ${\color{violet}N}={\color{red}m}{\color{purple}n}$}, \\
	0, & \text{otherwise}.
\end{cases}
\end{align*}
\end{Theorem}

\begin{Theorem}\label{th:main:Det2}
For ${\color{red}m}\geq 2$ we have
\begin{align*}
{\color{purple}\Delta}^{(2{\color{red}m})}_{{\color{violet}i}+{\color{blue}j}+2-{\color{red}m}} 
&=\begin{cases}
 (-1)^{{\color{red}m}{\color{purple}n}}{\color{cyan}\xi}_1    {\color{cyan}t}^{{\color{red}m}^2{\color{purple}n}({\color{purple}n}-1)/2} [{\color{purple}n}+1]_{{\color{cyan}t}^{\color{red}m}} , & \text{if ${\color{violet}N}={\color{red}m}{\color{purple}n}$}, \\
	(-1)^{{\color{red}m}{\color{purple}n}-{\color{red}m}-1} {\color{cyan}\xi}_1 {\color{cyan}t}^{{\color{red}m}({\color{purple}n}-1)({\color{red}m}{\color{purple}n}-2)/2}  [{\color{purple}n}]_{{\color{cyan}t}^{\color{red}m}} , & \text{if ${\color{violet}N}={\color{red}m}{\color{purple}n}-1$}, \\
	0, & \text{otherwise}.
\end{cases}
\end{align*}
\end{Theorem}

\begin{Theorem}\label{th:main:Det3}
For ${\color{red}m}\geq 3$ we have
\begin{align*}
{\color{purple}\Delta}^{(2{\color{red}m})}_{{\color{violet}i}+{\color{blue}j}+3-{\color{red}m}} 
&=\begin{cases}
	{\color{cyan}\xi}_1  {\color{cyan}t}^{{\color{red}m}^2 {\color{purple}n}({\color{purple}n}-1)/2 }   [{\color{purple}n}+1]^2_{{\color{cyan}t}^{\color{red}m}}  , & \text{if ${\color{violet}N}={\color{red}m}{\color{purple}n}$}, \\
	(-1)^{{\color{red}m}-1}{\color{cyan}\xi}_1 {\color{cyan}t}^{{\color{red}m}({\color{purple}n}-1)({\color{red}m}{\color{purple}n}-2)/2} {\color{teal}R}({\color{red}m}; {\color{cyan}t}, {\color{purple}n}-1) , & \text{if ${\color{violet}N}={\color{red}m}{\color{purple}n}-1$} ,\\
	-{\color{cyan}\xi}_1  {\color{cyan}t}^{{\color{red}m}({\color{purple}n}-1)({\color{red}m}{\color{purple}n}-4)/2}  [{\color{purple}n}]^2_{{\color{cyan}t}^{\color{red}m}}   , & \text{if ${\color{violet}N}={\color{red}m}{\color{purple}n}-2 $}, \\
	0, & \text{otherwise}.
\end{cases}
\end{align*}
where
\begin{align}
{\color{teal}R}({\color{red}m}; {\color{cyan}t},{\color{purple}n})
&=
	{\color{red}m}[{\color{red}m}]_{\color{cyan}t}
	\left(
\sum_{{\color{violet}i}=0}^{{\color{purple}n}} \binom{{\color{violet}i}+2}{2}  {\color{cyan}t}^{{\color{red}m}{\color{violet}i}}
+\sum_{{\color{violet}i}={\color{purple}n}+1}^{2{\color{purple}n}} \binom{2{\color{purple}n}-{\color{violet}i}+2}{2} {\color{cyan}t}^{{\color{red}m}{\color{violet}i}}
\right).\label{def:Rn}
\end{align}

\end{Theorem}

We see that Theorem \ref{th:main:Det2} implies Theorem 7.4 of Cigler \eqref{th:Cigler7.4}
when ${\color{red}m}=2$, and Theorem \ref{th:main:Det3} 
confirmes his conjecture 7.6 \eqref{conj:Cigler} when ${\color{red}m}=3$.

\bigskip

Our method is essentially a “guess-and-prove” approach that relies on established techniques involving Hankel continued fractions. While the final forms of our theorems are quite simple and closed, the proof process itself forced us to guess and manipulate extremely large and intricate explicit expressions. For instance, the complete expression for ${\color{cyan}C}_{3{\color{blue}j}}$ given in Lemma \ref{lem:ABCq3} is
{
\scriptsize
\begin{align*}
	{\color{cyan}C}_{3{\color{blue}j}}&=
	-  {\color{magenta}q}^2 \Bigl(1+ \frac{[{\color{blue}j}]_{{\color{cyan}t}^{\color{red}m}}  [{\color{blue}j}+1]_{{\color{cyan}t}^{\color{red}m}} {\color{magenta}q}}{{\color{teal}R}({\color{red}m}; {\color{cyan}t}, {\color{blue}j}-1)} \Bigr)  
	 \sum_{{\color{violet}d}=0}^{{\color{red}m}-2} 
\frac{ (-{\color{magenta}q})^{\color{violet}d} {\color{red}m}  }{{\color{red}m}-{\color{violet}d}} \sum_{{\color{violet}i}=0}^{{\color{violet}d}} \binom{{\color{red}m}-1-{\color{violet}d}+{\color{violet}i}}{{\color{violet}i}} \binom{ {\color{red}m}-1-{\color{violet}i}}{{\color{violet}d}-{\color{violet}i}}  {\color{cyan}t}^{\color{violet}i}\\
& \quad -  \frac{{\color{teal}R}({\color{red}m}; {\color{cyan}t}, {\color{blue}j})}{[{\color{blue}j}+1]_{{\color{cyan}t}^{\color{red}m}}^2   }   (-{\color{magenta}q})^{{\color{red}m}+1}
 + \Bigl( {\color{red}m}[{\color{red}m}]_{\color{cyan}t} \frac{[{\color{blue}j}]_{{\color{cyan}t}^{\color{red}m}}  [{\color{blue}j}+1]_{{\color{cyan}t}^{\color{red}m}}  }{{\color{teal}R}({\color{red}m}; {\color{cyan}t}, {\color{blue}j}-1)} - \frac{(1+{\color{cyan}t}^{{\color{red}m}({\color{blue}j}+1)})}{[{\color{blue}j}+1]_{{\color{cyan}t}^{\color{red}m}}  } \Bigr) (-{\color{magenta}q})^{{\color{red}m}+2}\\
& \quad - \frac{{\color{cyan}t}^{{\color{red}m}{\color{blue}j}}}{{\color{teal}R}({\color{red}m}; {\color{cyan}t}, {\color{blue}j}-1)}  (-{\color{magenta}q})^{{\color{red}m}+3},
\end{align*}
}%
where ${\color{teal}R}({\color{red}m}; {\color{cyan}t}, {\color{purple}n})$ is defined in \eqref{def:Rn}. Observe that in this formula, the terms ${\color{teal}R}({\color{red}m}; {\color{cyan}t},{\color{purple}n})$ appear both in numerators and in denominators.

\medskip
To establish our main theorems, we use the Hankel continued fraction approach developed in \cite{Han2016Adv}. The underlying idea is recalled in Section~\ref{sec:hfrac}.  
Our strategy proceeds as follows:

\medskip

1. Derive a quadratic equation for each of the series
$$
\frac{({\color{cyan}C}({\color{cyan}t},{\color{magenta}q})-1)^{\color{red}m}}{{\color{magenta}q}^2} \quad \text{and} \quad
\frac{({\color{cyan}C}({\color{cyan}t},{\color{magenta}q})-1)^{\color{red}m}}{{\color{magenta}q}^3}.
$$
To this end, we obtain a general formula in Section \ref{sec:quadra}, Corollary~\ref{cor:F}.

\smallskip

2. Apply Algorithm \hbox{\tt NextABC} to the initial coefficients appearing in the quadratic equations derived above. This produces the initial terms of a sequence of six-tuples \eqref{sixtuples}
\begin{equation*}
({\color{teal}A}_{{\color{purple}n}+1}, {\color{magenta}B}_{{\color{purple}n}+1}, {\color{cyan}C}_{{\color{purple}n}+1}; {\color{olive}k}_{{\color{purple}n}}, {\color{purple}a}_{{\color{purple}n}}, {\color{cyan}D}_{{\color{purple}n}}).
\end{equation*}
Based on these initial terms, we formulate a conjectural closed-form description of the sequence; see Lemmas \ref{lem:ABCq2} and \ref{lem:ABCq3}.

\smallskip

3. Prove that the conjectured formulas are correct, i.e., that they satisfy the relations specified in Algorithm \hbox{\tt NextABC}, in Sections \ref{sec:Q2} and~\ref{sec:Q3}.  
Because the proofs are very long and technically involved, we reproduce only the more accessible part in this paper.  
The more intricate computational verifications are carried out by computer; the programs are available on my homepage at
$$\hbox{\tt https://irma.math.unistra.fr/\char126guoniu/narayana.html}$$
We emphasize that our program delivers rigorous symbolic proofs, rather than merely verifying the initial terms.
\medskip

4. Construct the Hankel continued fraction from the resulting sequence of quantities; see Section \ref{sec:hfrac}, Lemmas \ref{lem:HFrac:Q0}–\ref{lem:HFrac:Q3}.

\medskip

5. Finally, compute the Hankel determinants from the Hankel continued fractions. This is carried out in Section \ref{sec:proofmain}.


\section{Notations and properties}\label{sec:notations} 
Since the expressions we obtain are rather lengthy, we first introduce some notation to simplify the exposition.
\begin{Definition}\label{def:rho}
For ${\color{violet}d}=0,1,2,\ldots, {\color{red}m}$, define
\begin{align*}
{\color{cyan}\rho}({\color{red}m}; {\color{cyan}t}, {\color{violet}d})&=  
 \sum_{{\color{violet}i}=0}^{{\color{violet}d}} \frac{{\color{red}m}({\color{red}m}-{\color{violet}d})}{({\color{red}m}-{\color{violet}i})({\color{red}m}-{\color{violet}d}+{\color{violet}i})} \binom{{\color{red}m}-{\color{violet}d}+{\color{violet}i}}{{\color{violet}i}} \binom{ {\color{red}m}-{\color{violet}i}}{{\color{violet}d}-{\color{violet}i}}  {\color{cyan}t}^{\color{violet}i},\\
{\color{olive}S}({\color{red}m}; {\color{cyan}t}, {\color{purple}n})&=\frac{{\color{red}m}(1-{\color{cyan}t}^{\color{red}m})}{1-{\color{cyan}t}} \left(
\sum_{{\color{violet}i}=0}^{{\color{purple}n}} ({\color{violet}i}+1)^2 {\color{cyan}t}^{{\color{red}m}{\color{violet}i}}
+\sum_{{\color{violet}i}={\color{purple}n}+1}^{2{\color{purple}n}} (2{\color{purple}n}-{\color{violet}i}+1)^2 {\color{cyan}t}^{{\color{red}m}{\color{violet}i}}
\right),\\
{\color{magenta}\beta}({\color{red}m}; {\color{cyan}t}, {\color{magenta}q}) &=\sum_{{\color{violet}d}=0}^{{\color{red}m}} {\color{cyan}\rho}({\color{red}m}; {\color{cyan}t}, {\color{violet}d})(-{\color{magenta}q})^{\color{violet}d}\\
&=1 -{\color{red}m}({\color{cyan}t}+1){\color{magenta}q} + \bigl(\binom{{\color{red}m}}{2} +{\color{red}m}({\color{red}m}-2){\color{cyan}t} + \binom{{\color{red}m}}{2}{\color{cyan}t}^2\bigr){\color{magenta}q}^2 +\cdots\\
& \quad +  \frac{{\color{red}m}(1-{\color{cyan}t}^{\color{red}m})}{1-{\color{cyan}t}} (-{\color{magenta}q})^{{\color{red}m}-1} +  (1+{\color{cyan}t}^{\color{red}m})(-{\color{magenta}q})^{\color{red}m}.
\end{align*}

We also set
\begin{equation*}
{\color{teal}\alpha}({\color{red}m}; {\color{cyan}t},{\color{magenta}q}) =  \sum_{{\color{violet}d}=0}^{{\color{red}m}-2} {\color{cyan}\rho}({\color{red}m}; {\color{cyan}t},{\color{violet}d})(-{\color{magenta}q})^{\color{violet}d},
\end{equation*}
so that
$$
{\color{magenta}\beta}({\color{red}m}; {\color{cyan}t},{\color{magenta}q}) - {\color{teal}\alpha}({\color{red}m}; {\color{cyan}t},{\color{magenta}q}) = \frac{{\color{red}m}(1-{\color{cyan}t}^{\color{red}m})}{1-{\color{cyan}t}} (-{\color{magenta}q})^{{\color{red}m}-1} +  (1+{\color{cyan}t}^{\color{red}m})(-{\color{magenta}q})^{\color{red}m}.
$$
\end{Definition}
For brevity, we write ${\color{cyan}\rho}({\color{violet}d})={\color{cyan}\rho}({\color{red}m}; {\color{cyan}t},{\color{violet}d}), {\color{teal}R}({\color{purple}n})={\color{teal}R}({\color{red}m};{\color{cyan}t},{\color{purple}n}), {\color{olive}S}({\color{purple}n})={\color{olive}S}({\color{red}m};{\color{cyan}t},{\color{purple}n}), {\color{magenta}\beta}({\color{magenta}q})={\color{magenta}\beta}({\color{red}m}; {\color{cyan}t}, {\color{magenta}q}), {\color{teal}\alpha}({\color{magenta}q}) = {\color{teal}\alpha}({\color{red}m}; {\color{cyan}t},{\color{magenta}q})$. 

\medskip

{\bf Remark}.  
When ${\color{violet}d} = {\color{red}m}$, we obtain ${\color{cyan}\rho}({\color{red}m};{\color{cyan}t},{\color{red}m}) = 1 + {\color{cyan}t}^{\color{red}m}$. For ${\color{violet}d} = 0,1,2,\ldots, {\color{red}m}-1$, the function ${\color{cyan}\rho}({\color{red}m};{\color{cyan}t},{\color{violet}d})$ can be expressed as
\begin{equation*}
{\color{cyan}\rho}({\color{red}m}; {\color{cyan}t},{\color{violet}d}) =  
\frac{{\color{red}m}}{{\color{red}m}-{\color{violet}d}} \sum_{{\color{violet}i}=0}^{{\color{violet}d}} \binom{{\color{red}m}-1-{\color{violet}d}+{\color{violet}i}}{{\color{violet}i}} \binom{{\color{red}m}-1-{\color{violet}i}}{{\color{violet}d}-{\color{violet}i}} {\color{cyan}t}^{\color{violet}i}.
\end{equation*}

Recall that the {\it Lucas polynomials} are given by
${\color{cyan}L}_0({\color{olive}x},{\color{teal}s})=2$ and, for ${\color{red}m}\geq 1$,
\begin{equation*}
{\color{cyan}L}_{\color{red}m}({\color{olive}x},{\color{teal}s})= 
\sum_{{\color{violet}i}=0}^{\lfloor {\color{red}m}/2\rfloor} 
\binom{{\color{red}m}-{\color{violet}i}}{{\color{violet}i}} \frac{{\color{red}m}}{{\color{red}m}-{\color{violet}i}} 
{\color{teal}s}^{{\color{violet}i}} {\color{olive}x}^{{\color{red}m}-2{\color{violet}i}}.
\end{equation*}
They satisfy 
\begin{equation}\label{Lucas:rec}
{\color{cyan}L}_{\color{red}m}({\color{olive}x},{\color{teal}s}) = {\color{olive}x} {\color{cyan}L}_{{\color{red}m}-1}({\color{olive}x},{\color{teal}s}) + {\color{teal}s} {\color{cyan}L}_{{\color{red}m}-2} ({\color{olive}x},{\color{teal}s})
\end{equation}
with initial values ${\color{cyan}L}_0({\color{olive}x},{\color{teal}s})=2$ and ${\color{cyan}L}_1({\color{olive}x},{\color{teal}s})={\color{olive}x}$.

\begin{Lemma}\label{lem:beta:Lm}
We have
$$
{\color{magenta}\beta}({\color{red}m};{\color{cyan}t},{\color{magenta}q}) = {\color{cyan}L}_{\color{red}m}(1-{\color{magenta}q}-{\color{cyan}t}{\color{magenta}q},-{\color{cyan}t}{\color{magenta}q}^2).
$$
\end{Lemma}

\begin{proof}
Consider the right-hand side of the claimed identity:
\begin{align*}
\RHS 
&=\sum_{{\color{violet}i}=0}^{\lfloor {\color{red}m}/2\rfloor} 
\binom{{\color{red}m}-{\color{violet}i}}{{\color{violet}i}} \frac{{\color{red}m}}{{\color{red}m}-{\color{violet}i}} 
(1-{\color{magenta}q}-{\color{cyan}t}{\color{magenta}q})^{{\color{red}m}-2{\color{violet}i}}(-{\color{cyan}t}{\color{magenta}q}^2)^{{\color{violet}i}}\\
&=\sum_{{\color{violet}i}=0}^{\lfloor {\color{red}m}/2\rfloor} 
\binom{{\color{red}m}-{\color{violet}i}}{{\color{violet}i}} \frac{{\color{red}m}}{{\color{red}m}-{\color{violet}i}} 
\sum_{\ell=0}^{{\color{red}m}-2{\color{violet}i}} \binom{{\color{red}m}-2{\color{violet}i}}{\ell} (-(1+{\color{cyan}t}){\color{magenta}q})^\ell (-{\color{cyan}t}{\color{magenta}q}^2)^{{\color{violet}i}}.
\end{align*}
For $0\le {\color{violet}d}\le {\color{red}m}$ and $0\le {\color{blue}j}\le {\color{violet}d}$, we extract the coefficient of ${\color{magenta}q}^{\color{violet}d} {\color{cyan}t}^{\color{blue}j}$ from the above expression. First,
\begin{align*}
[{\color{magenta}q}^{\color{violet}d}]\RHS 
&=\sum_{{\color{violet}i}=0}^{\lfloor {\color{violet}d}/2\rfloor} 
\binom{{\color{red}m}-{\color{violet}i}}{{\color{violet}i}} \frac{{\color{red}m}}{{\color{red}m}-{\color{violet}i}} 
\binom{{\color{red}m}-2{\color{violet}i}}{{\color{violet}d}-2{\color{violet}i}} (-(1+{\color{cyan}t}))^{{\color{violet}d}-2{\color{violet}i}} (-{\color{cyan}t})^{{\color{violet}i}}\\
&=(-1)^{\color{violet}d}\sum_{{\color{violet}i}=0}^{\lfloor {\color{violet}d}/2\rfloor} 
\binom{{\color{red}m}-{\color{violet}i}}{{\color{violet}i}} \frac{{\color{red}m}}{{\color{red}m}-{\color{violet}i}} 
\binom{{\color{red}m}-2{\color{violet}i}}{{\color{violet}d}-2{\color{violet}i}} \sum_{\ell=0}^{{\color{violet}d}-2{\color{violet}i}} \binom{{\color{violet}d}-2{\color{violet}i}}{\ell} {\color{cyan}t}^\ell (-{\color{cyan}t})^{\color{violet}i}.
\end{align*}

To obtain $[{\color{magenta}q}^{\color{violet}d} {\color{cyan}t}^{\color{blue}j}]\RHS$, set $\ell = {\color{blue}j}-{\color{violet}i}$ in the last sum.
The condition $0\le \ell \le {\color{violet}d}-2{\color{violet}i}$ becomes
$0\le {\color{blue}j}-{\color{violet}i} \le {\color{violet}d}-2{\color{violet}i}$, i.e., ${\color{violet}i}\le {\color{blue}j}\le {\color{violet}d}-{\color{violet}i}$. Hence ${\color{violet}i}\le {\color{blue}j}$ and ${\color{violet}i}\le {\color{violet}d}-{\color{blue}j}$,
and in particular ${\color{violet}i}\le \lfloor {\color{violet}d}/2\rfloor$.
Therefore,
\begin{align}
[{\color{magenta}q}^{\color{violet}d} {\color{cyan}t}^{\color{blue}j}] \RHS 
&=(-1)^{\color{violet}d}\sum_{{\color{violet}i}=0}^{\min\{{\color{blue}j}, {\color{violet}d}-{\color{blue}j}\}} 
\binom{{\color{red}m}-{\color{violet}i}}{{\color{violet}i}} \frac{{\color{red}m}}{{\color{red}m}-{\color{violet}i}} 
\binom{{\color{red}m}-2{\color{violet}i}}{{\color{violet}d}-2{\color{violet}i}} \binom{{\color{violet}d}-2{\color{violet}i}}{{\color{blue}j}-{\color{violet}i}} (-1)^{\color{violet}i}. \label{RHS}
\end{align}

We now check that these coefficients coincide with the corresponding coefficients in ${\color{magenta}\beta}({\color{red}m};{\color{cyan}t},{\color{magenta}q})$.
For ${\color{violet}d}\le {\color{red}m}-1$, we have
\begin{align*}
[{\color{magenta}q}^{\color{violet}d} {\color{cyan}t}^{\color{blue}j}] \RHS 
&=(-1)^{\color{violet}d} \frac{ {\color{red}m}({\color{red}m}-1+{\color{blue}j}-{\color{violet}d})!}{{\color{blue}j}!({\color{red}m}-{\color{violet}d})!}  
\sum_{{\color{violet}i}=0}^{\min\{{\color{blue}j}, {\color{violet}d}-{\color{blue}j}\}} 
\binom{{\color{blue}j}}{{\color{violet}i}} 
\binom{{\color{red}m}-{\color{violet}i}-1}{{\color{violet}d}-{\color{violet}i}-{\color{blue}j}} (-1)^{\color{violet}i} \\
&=(-1)^{\color{violet}d} \frac{ {\color{red}m}({\color{red}m}-1+{\color{blue}j}-{\color{violet}d})!}{{\color{blue}j}!({\color{red}m}-{\color{violet}d})!}  
\binom{{\color{red}m}-1-{\color{blue}j}}{{\color{violet}d}-{\color{blue}j}}\\
&=[{\color{magenta}q}^{\color{violet}d} {\color{cyan}t}^{\color{blue}j}]{\color{cyan}\rho}({\color{red}m}; {\color{cyan}t},{\color{violet}d})(-{\color{magenta}q})^{\color{violet}d},
\end{align*}
	where the sum over ${\color{violet}i}$ is evaluated using ``{T}he {M}ethode of {C}oefficients"; see \cite{Gessel2008slide}.

If ${\color{violet}d}={\color{red}m}$, then from \eqref{RHS} we obtain $[{\color{magenta}q}^{\color{red}m} {\color{cyan}t}^0]\RHS = [{\color{magenta}q}^{\color{red}m} {\color{cyan}t}^{\color{red}m}]\RHS = (-1)^{\color{red}m}$.
For ${\color{violet}d}={\color{red}m}$ and $1\le {\color{blue}j}\le {\color{red}m}-1$, \eqref{RHS} specializes to
\begin{align*}
[{\color{magenta}q}^{\color{red}m} {\color{cyan}t}^{\color{blue}j}] \RHS 
&=(-1)^{\color{red}m} \frac{{\color{red}m}}{{\color{blue}j}}  
\sum_{{\color{violet}i}=0}^{\min\{{\color{blue}j}, {\color{red}m}-{\color{blue}j}\}} 
\binom{{\color{blue}j}}{{\color{violet}i}} 
\binom{{\color{red}m}-{\color{violet}i}-1}{{\color{red}m}-{\color{violet}i}-{\color{blue}j}} (-1)^{\color{violet}i} \\
&=(-1)^{\color{red}m} \frac{{\color{red}m}}{{\color{blue}j}}  
\binom{{\color{red}m}-1-{\color{blue}j}}{{\color{red}m}-{\color{blue}j}} = 0. \qedhere
\end{align*}
\end{proof}

Moreover, the quantities ${\color{olive}S}({\color{blue}j})$ and ${\color{teal}R}({\color{blue}j})$ are connected through the following relationships:
\begin{align*}
\frac{{\color{red}m}(1-{\color{cyan}t}^{\color{red}m})}{1-{\color{cyan}t}} 
&= \frac{2 {\color{teal}R}({\color{blue}j}) -{\color{olive}S}({\color{blue}j})   }{[{\color{blue}j}+1]_{{\color{cyan}t}^{\color{red}m}}^2}
	=  \frac{-	2{\color{cyan}t}^{\color{red}m} {\color{teal}R}({\color{blue}j}) +{\color{olive}S}({\color{blue}j}+1)}{[{\color{blue}j}+2]_{{\color{cyan}t}^{\color{red}m}}^2}\\
	&=  \frac{-	2{\color{cyan}t}^{\color{red}m} {\color{teal}R}({\color{blue}j}-1) +{\color{olive}S}({\color{blue}j})}{[{\color{blue}j}+1]_{{\color{cyan}t}^{\color{red}m}}^2}
=  \frac{{\color{teal}R}({\color{blue}j}) -	{\color{cyan}t}^{\color{red}m} {\color{teal}R}({\color{blue}j}-1) }{[{\color{blue}j}+1]_{{\color{cyan}t}^{\color{red}m}}^2},\\
{\color{olive}S}({\color{blue}j}) &= {\color{teal}R}({\color{blue}j}) + {\color{cyan}t}^{\color{red}m} {\color{teal}R}({\color{blue}j}-1).
\end{align*}
\medskip

Throughout this paper, we adopt the following rule, referred to as the ``{\it index convention}" \cite{Han2020Trans}.  
It stipulates that, in any expression defined by cases, each formula applies only to those integer indices that have not already appeared as special values. 
For instant, in Lemma \ref{lem:HFrac:Q0}, the formula for ${\color{red}v}_{2{\color{blue}j}+1}$ is applicable for ${\color{blue}j} \geq 1$, but not for ${\color{blue}j} = 0$, because ${\color{red}v}_1$ has already been explicitly specified earlier.


\section{Hankel continued fractions}\label{sec:hfrac} 

The Jacobi continued fraction is a useful tool for evaluating Hankel determinants when all of them are nonzero. Since, in our case, some Hankel determinants vanish, we instead have to use the so-called “Hankel continued fractions” developed in \cite{Han2016Adv}. For further references, see \cite{Krattenthaler1998, Cigler2013spe, Cigler2019, Han2020Trans}. We now briefly recall the definition and the basic properties of Hankel continued fractions from \cite{Han2016Adv}.

\begin{Definition}\label{def:super}
For each positive integer ${\color{violet}\delta}$, a {\it super continued fraction} associated with ${\color{violet}\delta}$, called {\it super ${\color{violet}\delta}$-fraction} for short, is defined to be a continued fraction of the following form
\begin{equation}\label{eq:super}
{\color{blue}F}({\color{magenta}q})=
\cFrac{{\color{red}v}_0 {\color{magenta}q}^{{\color{olive}k}_0}}{ 1+{\color{red}u}_1({\color{magenta}q}) {\color{magenta}q}} 
- \cFrac{ {\color{red}v}_1 {\color{magenta}q}^{{\color{olive}k}_0+{\color{olive}k}_1+{\color{violet}\delta}}}{ 1+{\color{red}u}_2({\color{magenta}q}){\color{magenta}q}}
- \cFrac{ {\color{red}v}_2 {\color{magenta}q}^{{\color{olive}k}_1+{\color{olive}k}_2+{\color{violet}\delta}}}{ 1+{\color{red}u}_3({\color{magenta}q}){\color{magenta}q}}
- {\cdots} 
\end{equation}
where ${\color{red}v}_{\color{blue}j}\not=0$ are constants, ${\color{olive}k}_{\color{blue}j}$ are nonnegative integers and ${\color{red}u}_{\color{blue}j}({\color{magenta}q})$ are polynomials of 
degree less than or equal to ${\color{olive}k}_{{\color{blue}j}-1}+{\color{violet}\delta}-2$. By convention, $0$ is of degree $-1$.
\end{Definition}

When ${\color{violet}\delta}=1$ (resp. ${\color{violet}\delta}=2$) and all ${\color{olive}k}_{\color{blue}j}=0$, 
the super ${\color{violet}\delta}$-fraction \eqref{eq:super} 
is the traditional ${\color{olive}S}$-fraction (resp. ${\color{olive}J}$-fraction). 
A super $2$-fraction is called {\it Hankel continued fraction}.

\begin{Theorem}\label{th:super2}
(i) Let ${\color{violet}\delta}$ be a positive integer.
Each super ${\color{violet}\delta}$-fraction defines
a power series, and conversely, for each power series ${\color{blue}F}({\color{magenta}q})$,
the super ${\color{violet}\delta}$-fraction expansion of ${\color{blue}F}({\color{magenta}q})$ exists
and is unique.
\smallskip
(ii) Let ${\color{blue}F}({\color{magenta}q})$ be a power series such that its ${\color{violet}H}$-fraction 
	is given by \eqref{eq:super} with ${\color{violet}\delta}=2$. 
Then, all non-vanishing Hankel determinants of ${\color{blue}F}({\color{magenta}q})$ are given by
	\begin{equation}\label{eq:HankelDetFundamental}
		{\color{violet}H}_{{\color{teal}s}_{\color{blue}j}}({\color{blue}F}({\color{magenta}q}))= (-1)^{{\color{red}\epsilon}_{\color{blue}j}} {\color{red}v}_0^{{\color{teal}s}_{\color{blue}j}} {\color{red}v}_1^{{\color{teal}s}_{\color{blue}j}-{\color{teal}s}_1} {\color{red}v}_2^{{\color{teal}s}_{\color{blue}j}-{\color{teal}s}_2} \cdots {\color{red}v}_{{\color{blue}j}-1}^{{\color{teal}s}_{\color{blue}j}-{\color{teal}s}_{{\color{blue}j}-1}},
	\end{equation}
where ${\color{red}\epsilon}_{\color{blue}j} = \sum_{{\color{violet}i}=0}^{{\color{blue}j}-1} {{\color{olive}k}_{\color{violet}i}({\color{olive}k}_{\color{violet}i}+1)/2}$
and
${\color{teal}s}_{\color{blue}j}={\color{olive}k}_0+{\color{olive}k}_1+\cdots + {\color{olive}k}_{{\color{blue}j}-1}+{\color{blue}j}$ for every ${\color{blue}j}\geq 0$.
\end{Theorem}
By applying Theorem \ref{th:super2}, the proof of our main results reduces to explicitly determining the corresponding Hankel continued fractions. This is exactly what we now proceed to do. 
The ${\color{violet}H}$-fraction will be given in the standard form
\begin{equation}\label{def:Hfrac}
	\HF({\color{magenta}q}; ({\color{olive}k}_{\color{blue}j}), ({\color{red}v}_{\color{blue}j}), ({\color{red}u}_{\color{blue}j})):=
\cFrac{{\color{red}v}_0 {\color{magenta}q}^{{\color{olive}k}_0}}{ 1+{\color{red}u}_1({\color{magenta}q}) {\color{magenta}q}} 
- \cFrac{ {\color{red}v}_1 {\color{magenta}q}^{{\color{olive}k}_0+{\color{olive}k}_1+2}}{ 1+{\color{red}u}_2({\color{magenta}q}){\color{magenta}q}}
- \cFrac{ {\color{red}v}_2 {\color{magenta}q}^{{\color{olive}k}_1+{\color{olive}k}_2+2}}{ 1+{\color{red}u}_3({\color{magenta}q}){\color{magenta}q}}
- {\cdots} 
\end{equation}
with explict values for ${\color{olive}k}_{\color{blue}j}, {\color{red}v}_{\color{blue}j}, {\color{red}u}_{\color{blue}j}$.

\medskip

\medskip

We are now prepared to write down the explicit ${\color{violet}H}$-fractions of
$$
({\color{blue}\gamma}({\color{cyan}t},{\color{magenta}q})-1)^{\color{red}m},\quad 
\frac{({\color{blue}\gamma}({\color{cyan}t},{\color{magenta}q})-1)^{\color{red}m}}{{\color{magenta}q}},\quad
\frac{({\color{blue}\gamma}({\color{cyan}t},{\color{magenta}q})-1)^{\color{red}m}}{{\color{magenta}q}^2},\quad
\frac{({\color{blue}\gamma}({\color{cyan}t},{\color{magenta}q})-1)^{\color{red}m}}{{\color{magenta}q}^3}.
$$

\begin{Lemma}\label{lem:HFrac:Q0}
For ${\color{red}m}\geq 1$, the power series ${({\color{blue}\gamma}({\color{cyan}t},{\color{magenta}q})-1)^{\color{red}m}}$ has
	the following ${\color{violet}H}$-fraction expansion:
\begin{equation*}
{({\color{blue}\gamma}({\color{cyan}t},{\color{magenta}q})-1)^{\color{red}m}} 
	= \HF(({\color{olive}k}_{\color{blue}j}), ({\color{red}v}_{\color{blue}j}), ({\color{red}u}_{\color{blue}j})),
\end{equation*}
where
${\color{olive}k}_0={\color{red}m},  {\color{olive}k}_{2{\color{blue}j}+1}={\color{red}m}-2,  {\color{olive}k}_{2{\color{blue}j}+2}=0, {\color{red}v}_0 =1,  {\color{red}v}_1 = {\color{cyan}t}^{\color{red}m}$, 
and
\begin{align*}
	{\color{red}v}_{2{\color{blue}j}+1}&=-{(-{\color{cyan}t})^{\color{red}m}  [{\color{blue}j}]_{{\color{cyan}t}^{\color{red}m}}}/{ [{\color{blue}j}+1]_{{\color{cyan}t}^{\color{red}m}}},\\
	{\color{red}v}_{2{\color{blue}j}+2}&={(-1)^{{\color{red}m}+1}  [{\color{blue}j}+2]_{{\color{cyan}t}^{\color{red}m}}}/{[{\color{blue}j}+1]_{{\color{cyan}t}^{\color{red}m}}};\\
	1+{\color{red}u}_{1}({\color{magenta}q}){\color{magenta}q}&={\color{magenta}\beta}({\color{red}m};{\color{cyan}t},{\color{magenta}q}),\\
	1+{\color{red}u}_{2{\color{blue}j}+2}({\color{magenta}q}){\color{magenta}q} &={\color{magenta}\beta}({\color{red}m};{\color{cyan}t},{\color{magenta}q}) - (-{\color{magenta}q})^{\color{red}m}(1+{\color{cyan}t}^{\color{red}m}),\\
	{\color{red}u}_{2{\color{blue}j}+1}({\color{magenta}q})&=0.
\end{align*}
\end{Lemma}

\begin{Lemma}\label{lem:HFrac:Q1}
For ${\color{red}m}\geq 1$, the power series ${({\color{blue}\gamma}({\color{cyan}t},{\color{magenta}q})-1)^{\color{red}m}/{\color{magenta}q}}$ has
	the following ${\color{violet}H}$-fraction expansion:
\begin{equation*}
\frac{({\color{blue}\gamma}({\color{cyan}t},{\color{magenta}q})-1)^{\color{red}m}}{{\color{magenta}q}}
	= \HF(({\color{olive}k}_{\color{blue}j}), ({\color{red}v}_{\color{blue}j}), ({\color{red}u}_{\color{blue}j})),
\end{equation*}
where
${\color{olive}k}_{\color{blue}j} ={\color{red}m}-1,
{\color{red}v}_0 =1,
{\color{red}v}_{\color{blue}j} ={\color{cyan}t}^{\color{red}m},
1+{\color{red}u}_{\color{blue}j}({\color{magenta}q}){\color{magenta}q} ={\color{magenta}\beta}({\color{red}m};{\color{cyan}t},{\color{magenta}q}) $.
\end{Lemma}

\begin{Lemma}\label{lem:HFrac:Q2}
For ${\color{red}m}\geq 2$, the power series ${({\color{blue}\gamma}({\color{cyan}t},{\color{magenta}q})-1)^{\color{red}m}/{\color{magenta}q}^2}$ has
	the following ${\color{violet}H}$-fraction expansion:
\begin{equation*}
\frac{({\color{blue}\gamma}({\color{cyan}t},{\color{magenta}q})-1)^{\color{red}m}}{{\color{magenta}q}^2}
	= \HF(({\color{olive}k}_{\color{blue}j}), ({\color{red}v}_{\color{blue}j}), ({\color{red}u}_{\color{blue}j})),
\end{equation*}
where
${\color{olive}k}_{2{\color{blue}j}}={\color{red}m}-2,  {\color{olive}k}_{2{\color{blue}j}+1}=0, {\color{red}v}_0 =1$, and
\begin{align*}
	{\color{red}v}_{2{\color{blue}j}}&=-{(-{\color{cyan}t})^{\color{red}m}  [{\color{blue}j}]_{{\color{cyan}t}^{\color{red}m}}}/{ [{\color{blue}j}+1]_{{\color{cyan}t}^{\color{red}m}}},\\
	{\color{red}v}_{2{\color{blue}j}+1}&={(-1)^{{\color{red}m}+1}  [{\color{blue}j}+2]_{{\color{cyan}t}^{\color{red}m}}}/{[{\color{blue}j}+1]_{{\color{cyan}t}^{\color{red}m}}},\\
	1+{\color{red}u}_{2{\color{blue}j}+1}({\color{magenta}q}){\color{magenta}q} &={\color{magenta}\beta}({\color{red}m};{\color{cyan}t},{\color{magenta}q}) - (-{\color{magenta}q})^{\color{red}m}(1+{\color{cyan}t}^{\color{red}m}),\\
	{\color{red}u}_{2{\color{blue}j}}({\color{magenta}q})&=0.
\end{align*}
\end{Lemma}

\begin{Lemma}\label{lem:HFrac:Q3}
For ${\color{red}m}\geq 3$, the power series ${({\color{blue}\gamma}({\color{cyan}t},{\color{magenta}q})-1)^{\color{red}m}/{\color{magenta}q}^3}$ has
	the following ${\color{violet}H}$-fraction expansion:
\begin{equation*}
\frac{({\color{blue}\gamma}({\color{cyan}t},{\color{magenta}q})-1)^{\color{red}m}}{{\color{magenta}q}^3}
	= \HF(({\color{olive}k}_{\color{blue}j}), ({\color{red}v}_{\color{blue}j}), ({\color{red}u}_{\color{blue}j})),
\end{equation*}
where
${\color{olive}k}_{3{\color{blue}j}}={\color{red}m}-3, 
{\color{olive}k}_{3{\color{blue}j}+1}=0, 
{\color{olive}k}_{3{\color{blue}j}+2}=0,
{\color{red}v}_0=-1$, and
\begin{align*}
{\color{red}v}_{3{\color{blue}j}} &=(-{\color{cyan}t})^{{\color{red}m}+1}   { {\color{teal}R}({\color{blue}j}-1) }/{ [{\color{blue}j}+1]_{{\color{cyan}t}^{\color{red}m}}^{2}  },\\
{\color{red}v}_{3{\color{blue}j}+1} &={(-1)^{{\color{red}m}}  {\color{teal}R}({\color{blue}j}) }/{[{\color{blue}j}+1]_{{\color{cyan}t}^{\color{red}m}}^2  },\\
{\color{red}v}_{3{\color{blue}j}+2} &=- [{\color{blue}j}+1]_{{\color{cyan}t}^{\color{red}m}}^2 [{\color{blue}j}+2]_{{\color{cyan}t}^{\color{red}m}}^2  /{{\color{teal}R}({\color{blue}j})^2 },\\
{\color{red}u}_{3{\color{blue}j}}({\color{magenta}q}) &=- {[{\color{blue}j}]_{{\color{cyan}t}^{\color{red}m}} [{\color{blue}j}+1]_{{\color{cyan}t}^{\color{red}m}}  }/{{\color{teal}R}({\color{blue}j}-1) },\\
1+{\color{red}u}_{3{\color{blue}j}+1}({\color{magenta}q}){\color{magenta}q} &={\color{teal}\alpha}({\color{magenta}q}) ,\\
{\color{red}u}_{3{\color{blue}j}+2}({\color{magenta}q}) &={[{\color{blue}j}+1]_{{\color{cyan}t}^{\color{red}m}} [{\color{blue}j}+2]_{{\color{cyan}t}^{\color{red}m}}  }/{{\color{teal}R}({\color{blue}j}) }.
\end{align*}
\end{Lemma}


\section{Proofs of the main theorems}\label{sec:proofmain} 

In this section we prove the four main theorems using the four explicit ${\color{violet}H}$-fraction given in the previous section.

\begin{proof}[Proof of Theorem \ref{th:main:Det0}]
We apply the second part of Theorem \ref{th:super2} to the ${\color{violet}H}$-fraction from Lemma \ref{lem:HFrac:Q0} that corresponds to ${({\color{blue}\gamma}({\color{cyan}t},{\color{magenta}q})-1)^{\color{red}m}}$. In this setting, we obtain ${\color{teal}s}_0=0$,
\begin{align*}
{\color{teal}s}_{2{\color{blue}j}+1}
&={\color{olive}k}_0+{\color{olive}k}_1+\cdots + {\color{olive}k}_{2{\color{blue}j}}+2{\color{blue}j}+1
= {\color{red}m}({\color{blue}j}+1)+1,\\
{\color{teal}s}_{2{\color{blue}j}+2}&={\color{teal}s}_{2{\color{blue}j}+1} +  {\color{olive}k}_{2{\color{blue}j}+1}+1
= {\color{red}m}({\color{blue}j}+1)+1 + ({\color{red}m}-2) +1
= {\color{red}m}({\color{blue}j}+2),
\end{align*}
and ${\color{red}\epsilon}_0 =0$,
\begin{align*}
{\color{red}\epsilon}_{2{\color{blue}j}+1} 
&={\color{red}m}({\color{red}m}+1)/2 +   \sum_{{\color{violet}i}=1}^{2{\color{blue}j}} {{\color{olive}k}_{\color{violet}i}({\color{olive}k}_{\color{violet}i}+1)/2}\\
	&= {\color{red}m}({\color{red}m}+1)/2 +   {\color{blue}j} ({\color{red}m}-2)({\color{red}m}-1)/2,\\
{\color{red}\epsilon}_{2{\color{blue}j}+2}  
&={\color{red}\epsilon}_{2{\color{blue}j}+1} + ({\color{red}m}-2)({\color{red}m}-1)/2\\
	&= {\color{red}m}({\color{red}m}+1)/2 +   ({\color{blue}j}+1) ({\color{red}m}-2)({\color{red}m}-1)/2.
\end{align*}
Since ${\color{red}v}_0=1$ and ${\color{red}v}_{2{\color{violet}i}+2} {\color{red}v}_{2{\color{violet}i}+3}  ={\color{cyan}t}^{{\color{red}m}}$ , we have
${\color{violet}H}_0 =1,
{\color{violet}H}_{{\color{teal}s}_1} =(-1)^{{\color{red}m}({\color{red}m}+1)/2}$,  
\begin{align*}
{\color{violet}H}_{{\color{teal}s}_{2{\color{blue}j}+1}}({\color{blue}F}({\color{magenta}q}))
&=(-1)^{{\color{red}\epsilon}_{2{\color{blue}j}+1}}  {\color{red}v}_1^{{\color{teal}s}_{2{\color{blue}j}+1}-{\color{teal}s}_1} {\color{red}v}_2^{{\color{teal}s}_{2{\color{blue}j}+1}-{\color{teal}s}_2} \cdots {\color{red}v}_{2{\color{blue}j}}^{{\color{teal}s}_{2{\color{blue}j}+1}-{\color{teal}s}_{2{\color{blue}j}}},\\
&=(-1)^{{\color{red}\epsilon}_{2{\color{blue}j}+1}}   {\color{red}v}_1^{{\color{teal}s}_{2{\color{blue}j}+1}-{\color{teal}s}_1} {\color{red}v}_{2{\color{blue}j}}^{{\color{teal}s}_{2{\color{blue}j}+1}-{\color{teal}s}_{2{\color{blue}j}}}\prod_{{\color{violet}i}=0}^{{\color{blue}j}-2} {\color{red}v}_{2{\color{violet}i}+2}^{{\color{teal}s}_{2{\color{blue}j}+1}-{\color{teal}s}_{2{\color{violet}i}+2}} {\color{red}v}_{2{\color{violet}i}+3}^{{\color{teal}s}_{2{\color{blue}j}+1}-{\color{teal}s}_{2{\color{violet}i}+3}}\\
&=(-1)^{  ({\color{blue}j}+1)({\color{red}m}-2)({\color{red}m}-1)/2 +{\color{red}m} +({\color{red}m}+1)({\color{blue}j}-1) }  
	{\color{cyan}t}^{{\color{red}m}^2{\color{blue}j}}\\
& \quad \times  \frac{[{\color{blue}j}+1]_{{\color{cyan}t}^{\color{red}m}}}{[{\color{blue}j}]_{{\color{cyan}t}^{\color{red}m}}} 
\prod_{{\color{violet}i}=0}^{{\color{blue}j}-2} {\color{cyan}t}^{{\color{red}m}({\color{red}m}{\color{blue}j}-{\color{red}m}+1) -{\color{violet}i}{\color{red}m}^2 -{\color{red}m}}
 \prod_{{\color{violet}i}=0}^{{\color{blue}j}-2}   \frac  { [{\color{violet}i}+2]_{{\color{cyan}t}^{\color{red}m}}}{  [{\color{violet}i}+1]_{{\color{cyan}t}^{\color{red}m}}   }\\
&=(-1)^{  ({\color{blue}j}+1){\color{red}m}({\color{red}m}-1)/2 +{\color{red}m}   }   
{\color{cyan}t}^{ {\color{red}m}^2 {\color{blue}j} ({\color{blue}j}+1)/2 } [{\color{blue}j}+1]_{{\color{cyan}t}^{\color{red}m}},
\end{align*}
and
\begin{align*}
{\color{violet}H}_{{\color{teal}s}_{2{\color{blue}j}+2}}({\color{blue}F}({\color{magenta}q}))
&=(-1)^{{\color{red}\epsilon}_{2{\color{blue}j}+2}}  {\color{red}v}_1^{{\color{teal}s}_{2{\color{blue}j}+2}-{\color{teal}s}_1} {\color{red}v}_2^{{\color{teal}s}_{2{\color{blue}j}+2}-{\color{teal}s}_2} \cdots {\color{red}v}_{2{\color{blue}j}+1}^{{\color{teal}s}_{2{\color{blue}j}+2}-{\color{teal}s}_{2{\color{blue}j}+1}},\\
&=(-1)^{{\color{red}\epsilon}_{2{\color{blue}j}+2}}  {\color{red}v}_1^{{\color{teal}s}_{2{\color{blue}j}+2}-{\color{teal}s}_1}  \prod_{{\color{violet}i}=0}^{{\color{blue}j}-1} {\color{red}v}_{2{\color{violet}i}+2}^{{\color{teal}s}_{2{\color{blue}j}+2}-{\color{teal}s}_{2{\color{violet}i}+2}} {\color{red}v}_{2{\color{violet}i}+3}^{{\color{teal}s}_{2{\color{blue}j}+2}-{\color{teal}s}_{2{\color{violet}i}+3}}\\
&=(-1)^{{\color{red}m}({\color{red}m}+1)/2 + ({\color{blue}j}+1)({\color{red}m}-2)({\color{red}m}-1)/2  }  {\color{cyan}t}^{{\color{red}m}({\color{red}m}{\color{blue}j}+{\color{red}m}-1)}\\
& \quad \times\prod_{{\color{violet}i}=0}^{{\color{blue}j}-1} ({\color{red}v}_{2{\color{violet}i}+2} {\color{red}v}_{2{\color{violet}i}+3})^{{\color{red}m}({\color{blue}j}-{\color{violet}i})  } \prod_{{\color{violet}i}=0}^{{\color{blue}j}-1} {\color{red}v}_{2{\color{violet}i}+3}^{ -1}\\
&=(-1)^{{\color{red}m}({\color{red}m}+1)/2 + ({\color{blue}j}+1)({\color{red}m}-2)({\color{red}m}-1)/2  }  {\color{cyan}t}^{{\color{red}m}({\color{red}m}{\color{blue}j}+{\color{red}m}-1)}\\
& \quad \times\prod_{{\color{violet}i}=0}^{{\color{blue}j}-1} ({\color{cyan}t}^{\color{red}m})^{{\color{red}m}({\color{blue}j}-{\color{violet}i})  }
\prod_{{\color{violet}i}=0}^{{\color{blue}j}-1} \frac{(-1)^{{\color{red}m}+1} [{\color{violet}i}+2]_{{\color{cyan}t}^{\color{red}m}}}{{\color{cyan}t}^{\color{red}m} [{\color{violet}i}+1]_{{\color{cyan}t}^{\color{red}m}}}\\
&=(-1)^{({\color{blue}j}+2){\color{red}m}({\color{red}m}-1)/2 +1 }  {\color{cyan}t}^{{\color{red}m}({\color{blue}j}+1)({\color{red}m}{\color{blue}j}+2{\color{red}m}-2)/2 } 
{ [{\color{blue}j}+1]_{{\color{cyan}t}^{\color{red}m}}}. \qedhere
\end{align*}
\end{proof}

\begin{proof}[Proof of Theorem \ref{th:main:Det1}]
Apply the second part of Theorem \ref{th:super2} to the ${\color{violet}H}$-fraction from Lemma \ref{lem:HFrac:Q1} corresponding to ${({\color{blue}\gamma}({\color{cyan}t},{\color{magenta}q})-1)^{\color{red}m}}/{{\color{magenta}q}}$. Then we obtain
\begin{align*}
{\color{teal}s}_0 &=0, \qquad 
{\color{teal}s}_{\color{blue}j}
={\color{olive}k}_0+{\color{olive}k}_1+\cdots + {\color{olive}k}_{{\color{blue}j}-1}+{\color{blue}j}
={\color{blue}j}({\color{red}m}-1) +{\color{blue}j}
= {\color{red}m}{\color{blue}j},
\end{align*}
and
\begin{equation*}
{\color{red}\epsilon}_{\color{blue}j} 
= \sum_{{\color{violet}i}=0}^{{\color{blue}j}-1} {{\color{olive}k}_{\color{violet}i}({\color{olive}k}_{\color{violet}i}+1)/2}
= {\color{blue}j} {\color{red}m}({\color{red}m}-1)/2.
\end{equation*}
Since ${\color{red}v}_0=1$, we have ${\color{violet}H}_0 =1$,
\begin{align*}
{\color{violet}H}_{{\color{teal}s}_{\color{blue}j}}({\color{blue}F}({\color{magenta}q}))
&=(-1)^{{\color{red}\epsilon}_{\color{blue}j}} {\color{red}v}_0^{{\color{teal}s}_{\color{blue}j}} {\color{red}v}_1^{{\color{teal}s}_{\color{blue}j}-{\color{teal}s}_1} {\color{red}v}_2^{{\color{teal}s}_{\color{blue}j}-{\color{teal}s}_2} \cdots {\color{red}v}_{{\color{blue}j}-1}^{{\color{teal}s}_{\color{blue}j}-{\color{teal}s}_{{\color{blue}j}-1}}\\
&=(-1)^{{\color{blue}j}{\color{red}m}({\color{red}m}-1)/2}  {\color{red}v}_1^{{\color{blue}j}{\color{red}m}-{\color{red}m}} {\color{red}v}_2^{{\color{blue}j}{\color{red}m}-2{\color{red}m}} \cdots {\color{red}v}_{{\color{blue}j}-1}^{{\color{blue}j}{\color{red}m}-({\color{blue}j}-1){\color{red}m}}\\
&=(-1)^{{\color{blue}j}{\color{red}m}({\color{red}m}-1)/2}  ({\color{cyan}t}^{\color{red}m})^{{\color{blue}j}{\color{red}m}-{\color{red}m}+{\color{blue}j}{\color{red}m}-2{\color{red}m} + \cdots   +{\color{blue}j}{\color{red}m}-({\color{blue}j}-1){\color{red}m}}\\
&=(-1)^{{\color{blue}j}{\color{red}m}({\color{red}m}-1)/2}  {\color{cyan}t}^{{\color{red}m}^2{\color{blue}j}({\color{blue}j}-1)/2} \qedhere
\end{align*}
\end{proof}

\begin{proof}[Proof of Theorem \ref{th:main:Det2}]
We apply the second part of Theorem \ref{th:super2} to the ${\color{violet}H}$-fraction from Lemma \ref{lem:HFrac:Q2}, corresponding to ${({\color{blue}\gamma}({\color{cyan}t},{\color{magenta}q})-1)^{\color{red}m}}/{{\color{magenta}q}^2}$. This yields
${\color{teal}s}_0 =0$,
\begin{align*}
{\color{teal}s}_{2{\color{blue}j}+1}
&={\color{olive}k}_0+{\color{olive}k}_1+\cdots + {\color{olive}k}_{2{\color{blue}j}}+2{\color{blue}j}+1\\
	&=({\color{blue}j}+1)({\color{red}m}-2) + 2{\color{blue}j}+1
	=({\color{blue}j}+1){\color{red}m} -1 ,\\
{\color{teal}s}_{2{\color{blue}j}+2}&={\color{olive}k}_0+{\color{olive}k}_1+\cdots + {\color{olive}k}_{2{\color{blue}j}+1}+2{\color{blue}j}+2\\
&=({\color{olive}k}_0+{\color{olive}k}_1+\cdots + {\color{olive}k}_{2{\color{blue}j}}+2{\color{blue}j}+1) +1
=({\color{blue}j}+1){\color{red}m},
\end{align*}
and ${\color{red}\epsilon}_0 =0$,
\begin{align*}
{\color{red}\epsilon}_{2{\color{blue}j}+1} &=\sum_{{\color{violet}i}=0}^{2{\color{blue}j}} {{\color{olive}k}_{\color{violet}i}({\color{olive}k}_{\color{violet}i}+1)/2}
=({\color{blue}j}+1) ({\color{red}m}-2)({\color{red}m}-1)/2,\\
{\color{red}\epsilon}_{2{\color{blue}j}+2}  &=({\color{blue}j}+1) ({\color{red}m}-2)({\color{red}m}-1)/2.
\end{align*}
Since ${\color{red}v}_0=1$ and  ${\color{red}v}_{2{\color{violet}i}+1} {\color{red}v}_{2{\color{violet}i}+2}  ={\color{cyan}t}^{\color{red}m}$, we have
${\color{violet}H}_0 =1$,
\begin{align*}
{\color{violet}H}_{{\color{teal}s}_{2{\color{blue}j}+1}}({\color{blue}F}({\color{magenta}q}))
&=(-1)^{{\color{red}\epsilon}_{2{\color{blue}j}+1}}  {\color{red}v}_1^{{\color{teal}s}_{2{\color{blue}j}+1}-{\color{teal}s}_1} {\color{red}v}_2^{{\color{teal}s}_{2{\color{blue}j}+1}-{\color{teal}s}_2} \cdots {\color{red}v}_{2{\color{blue}j}}^{{\color{teal}s}_{2{\color{blue}j}+1}-{\color{teal}s}_{2{\color{blue}j}}},\\
&={\color{cyan}\xi}_2    \prod_{{\color{violet}i}=0}^{{\color{blue}j}-1} {\color{red}v}_{2{\color{violet}i}+1}^{({\color{blue}j}-{\color{violet}i}){\color{red}m}} {\color{red}v}_{2{\color{violet}i}+2}^{ ({\color{blue}j}-{\color{violet}i}){\color{red}m}-1 }\\
&={\color{cyan}\xi}_2  
\prod_{{\color{violet}i}=0}^{{\color{blue}j}-1} \left({\color{red}v}_{2{\color{violet}i}+1} {\color{red}v}_{2{\color{violet}i}+2} \right)^{({\color{blue}j}-{\color{violet}i}){\color{red}m}}
\prod_{{\color{violet}i}=0}^{{\color{blue}j}-1}  {\color{red}v}_{2{\color{violet}i}+2}^{ -1 }\\
&={\color{cyan}\xi}_2  
\prod_{{\color{violet}i}=0}^{{\color{blue}j}-1} \left( {\color{cyan}t}^{\color{red}m}\right)^{({\color{blue}j}-{\color{violet}i}){\color{red}m}}
/ \prod_{{\color{violet}i}=0}^{{\color{blue}j}-1}   \frac{ (-1)^{{\color{red}m}+1} {\color{cyan}t}^{{\color{red}m}}[{\color{violet}i}+1]_{{\color{cyan}t}^{\color{red}m}}   }{ [{\color{violet}i}+2]_{{\color{cyan}t}^{\color{red}m}}}\\
&=(-1)^{({\color{red}m}+1)({\color{red}m}{\color{blue}j}+{\color{red}m}-2)/2}  {\color{cyan}t}^{{\color{red}m}{\color{blue}j}({\color{red}m}{\color{blue}j}+{\color{red}m}-2)/2}  [{\color{blue}j}+1]_{{\color{cyan}t}^{\color{red}m}},
\end{align*}
where ${\color{cyan}\xi}_2 =(-1)^{({\color{blue}j}+1)({\color{red}m}-2)({\color{red}m}-1)/2}$,
and
\begin{align*}
{\color{violet}H}_{{\color{teal}s}_{2{\color{blue}j}+2}}({\color{blue}F}({\color{magenta}q}))
&=(-1)^{{\color{red}\epsilon}_{2{\color{blue}j}+2}}  {\color{red}v}_1^{{\color{teal}s}_{2{\color{blue}j}+2}-{\color{teal}s}_1} {\color{red}v}_2^{{\color{teal}s}_{2{\color{blue}j}+2}-{\color{teal}s}_2} \cdots {\color{red}v}_{2{\color{blue}j}+1}^{{\color{teal}s}_{2{\color{blue}j}+2}-{\color{teal}s}_{2{\color{blue}j}+1}},\\
&=(-1)^{{\color{red}\epsilon}_{2{\color{blue}j}+2}}   \prod_{{\color{violet}i}=0}^{{\color{blue}j}-1} {\color{red}v}_{2{\color{violet}i}+1}^{{\color{teal}s}_{2{\color{blue}j}+2}-{\color{teal}s}_{2{\color{violet}i}+1}} {\color{red}v}_{2{\color{violet}i}+2}^{{\color{teal}s}_{2{\color{blue}j}+2}-{\color{teal}s}_{2{\color{violet}i}+2}}   \times {\color{red}v}_{2{\color{blue}j}+1}^{{\color{teal}s}_{2{\color{blue}j}+2}-{\color{teal}s}_{2{\color{blue}j}+1}}\\
&={\color{cyan}\xi}_2  {\color{red}v}_{2{\color{blue}j}+1} 
\prod_{{\color{violet}i}=0}^{{\color{blue}j}-1} \left({\color{red}v}_{2{\color{violet}i}+1} {\color{red}v}_{2{\color{violet}i}+2}  \right)^{({\color{blue}j}-{\color{violet}i}){\color{red}m}+1}
\prod_{{\color{violet}i}=0}^{{\color{blue}j}-1}  {\color{red}v}_{2{\color{violet}i}+2}^{ -1 }\\
&={\color{cyan}\xi}_2 {\color{red}v}_{2{\color{blue}j}+1} 
\prod_{{\color{violet}i}=0}^{{\color{blue}j}-1} \left( {\color{cyan}t}^{\color{red}m}\right)^{({\color{blue}j}-{\color{violet}i}){\color{red}m}+1}
/ \prod_{{\color{violet}i}=0}^{{\color{blue}j}-1}   \frac{ (-1)^{{\color{red}m}+1} {\color{cyan}t}^{{\color{red}m}}[{\color{violet}i}+1]_{{\color{cyan}t}^{\color{red}m}}   }{ [{\color{violet}i}+2]_{{\color{cyan}t}^{\color{red}m}}}\\
&=(-1)^{({\color{blue}j}+1){\color{red}m}({\color{red}m}+1)/2}    {\color{cyan}t}^{{\color{red}m}^2{\color{blue}j}({\color{blue}j}+1)/2} [{\color{blue}j}+2]_{{\color{cyan}t}^{\color{red}m}}. \qedhere
\end{align*}
\end{proof}

\begin{proof}[Proof of Theorem \ref{th:main:Det3}]
We apply the second part of Theorem \ref{th:super2} to the ${\color{violet}H}$-fraction from Lemma \ref{lem:HFrac:Q3} corresponding to ${({\color{blue}\gamma}({\color{cyan}t},{\color{magenta}q})-1)^{\color{red}m}}/{{\color{magenta}q}^3}$. In this setting, we have ${\color{teal}s}_0 = 0$,
\begin{align*}
{\color{teal}s}_{3{\color{blue}j}+1}
&={\color{olive}k}_0+{\color{olive}k}_1+\cdots + {\color{olive}k}_{3{\color{blue}j}}+3{\color{blue}j}+1\\
	&=({\color{blue}j}+1)({\color{red}m}-3) + 3{\color{blue}j}+1
=({\color{blue}j}+1){\color{red}m} -2,\\
{\color{teal}s}_{3{\color{blue}j}+2}&={\color{teal}s}_{3{\color{blue}j}+1} +  {\color{olive}k}_{3{\color{blue}j}+1}+1
=({\color{blue}j}+1){\color{red}m} -1,\\
{\color{teal}s}_{3{\color{blue}j}+3}&={\color{teal}s}_{3{\color{blue}j}+2} + {\color{olive}k}_{3{\color{blue}j}+2}+1
=({\color{blue}j}+1){\color{red}m},
\end{align*}
and ${\color{red}\epsilon}_0 =0$,
\begin{align*}
	{\color{red}\epsilon}_{3{\color{blue}j}+1} & = {\color{red}\epsilon}_{3{\color{blue}j}+2} = {\color{red}\epsilon}_{3{\color{blue}j}+3}\\
	&=\sum_{{\color{violet}i}=0}^{3{\color{blue}j}} {{\color{olive}k}_{\color{violet}i}({\color{olive}k}_{\color{violet}i}+1)/2} =({\color{blue}j}+1) ({\color{red}m}-3)({\color{red}m}-2)/2.
\end{align*}
We can verify that
\begin{align*}
{\color{red}v}_{3{\color{violet}i}+1} {\color{red}v}_{3{\color{violet}i}+2} {\color{red}v}_{3{\color{violet}i}+3} &={\color{cyan}t}^{\color{red}m},\\
{\color{red}v}_{3{\color{violet}i}+2} {\color{red}v}_{3{\color{violet}i}+3}^2 &=- {\color{cyan}t}^{2{\color{red}m}}[{\color{violet}i}+1]^2_{{\color{cyan}t}^{\color{red}m}}   / [{\color{violet}i}+2]^2_{{\color{cyan}t}^{\color{red}m}},\\
{\color{red}v}_{3{\color{blue}j}+1}^2 {\color{red}v}_{3{\color{blue}j}+2}  
&=-     [{\color{blue}j}+2]^2_{{\color{cyan}t}^{\color{red}m}} / [{\color{blue}j}+1]^2_{{\color{cyan}t}^{\color{red}m}}.
\end{align*}
Since ${\color{red}v}_0=1$, we have ${\color{violet}H}_0 =1$,
\begin{align*}
{\color{violet}H}_{{\color{teal}s}_{3{\color{blue}j}+1}}({\color{blue}F}({\color{magenta}q}))
&=(-1)^{{\color{red}\epsilon}_{3{\color{blue}j}+1}}  {\color{red}v}_1^{{\color{teal}s}_{3{\color{blue}j}+1}-{\color{teal}s}_1} {\color{red}v}_2^{{\color{teal}s}_{3{\color{blue}j}+1}-{\color{teal}s}_2} \cdots {\color{red}v}_{3{\color{blue}j}}^{{\color{teal}s}_{3{\color{blue}j}+1}-{\color{teal}s}_{3{\color{blue}j}}},\\
&=(-1)^{{\color{red}\epsilon}_{3{\color{blue}j}+1}}   \prod_{{\color{violet}i}=0}^{{\color{blue}j}-1} {\color{red}v}_{3{\color{violet}i}+1}^{{\color{teal}s}_{3{\color{blue}j}+1}-{\color{teal}s}_{3{\color{violet}i}+1}} {\color{red}v}_{3{\color{violet}i}+2}^{{\color{teal}s}_{3{\color{blue}j}+1}-{\color{teal}s}_{3{\color{violet}i}+2}} {\color{red}v}_{3{\color{violet}i}+3}^{{\color{teal}s}_{3{\color{blue}j}+1}-{\color{teal}s}_{3{\color{violet}i}+3}}\\
&={\color{cyan}\xi}_3   \prod_{{\color{violet}i}=0}^{{\color{blue}j}-1} {\color{red}v}_{3{\color{violet}i}+1}^{({\color{blue}j}-{\color{violet}i}){\color{red}m}} {\color{red}v}_{3{\color{violet}i}+2}^{ ({\color{blue}j}-{\color{violet}i}){\color{red}m}-1 } {\color{red}v}_{3{\color{violet}i}+3}^{({\color{blue}j}-{\color{violet}i}){\color{red}m}-2}\\
&={\color{cyan}\xi}_3  
\prod_{{\color{violet}i}=0}^{{\color{blue}j}-1} \left({\color{red}v}_{3{\color{violet}i}+1} {\color{red}v}_{3{\color{violet}i}+2} {\color{red}v}_{3{\color{violet}i}+3} )\right)^{({\color{blue}j}-{\color{violet}i}){\color{red}m}}
\prod_{{\color{violet}i}=0}^{{\color{blue}j}-1}  {\color{red}v}_{3{\color{violet}i}+2}^{ -1 } {\color{red}v}_{3{\color{violet}i}+3}^{-2}\\
&={\color{cyan}\xi}_3  
\prod_{{\color{violet}i}=0}^{{\color{blue}j}-1} \left( {\color{cyan}t}^{\color{red}m}\right)^{({\color{blue}j}-{\color{violet}i}){\color{red}m}}
/ \prod_{{\color{violet}i}=0}^{{\color{blue}j}-1}   \frac{ - {\color{cyan}t}^{2{\color{red}m}}[{\color{violet}i}+1]^2_{{\color{cyan}t}^{\color{red}m}}   }{ [{\color{violet}i}+2]^2_{{\color{cyan}t}^{\color{red}m}}}\\
&=(-1)^{1+({\color{blue}j}+1){\color{red}m}({\color{red}m}-1)/2}  {\color{cyan}t}^{{\color{red}m}{\color{blue}j}({\color{red}m}{\color{blue}j}+{\color{red}m}-4)/2}  [{\color{blue}j}+1]^2_{{\color{cyan}t}^{\color{red}m}},
\end{align*}
where ${\color{cyan}\xi}_3 =(-1)^{({\color{blue}j}+1)({\color{red}m}-3)({\color{red}m}-2)/2}$,
and
\begin{align*}
{\color{violet}H}_{{\color{teal}s}_{3{\color{blue}j}+2}}({\color{blue}F}({\color{magenta}q}))
&=(-1)^{{\color{red}\epsilon}_{3{\color{blue}j}+2}}  {\color{red}v}_1^{{\color{teal}s}_{3{\color{blue}j}+2}-{\color{teal}s}_1} {\color{red}v}_2^{{\color{teal}s}_{3{\color{blue}j}+2}-{\color{teal}s}_2} \cdots {\color{red}v}_{3{\color{blue}j}+1}^{{\color{teal}s}_{3{\color{blue}j}+2}-{\color{teal}s}_{3{\color{blue}j}+1}},\\
&=(-1)^{{\color{red}\epsilon}_{3{\color{blue}j}+2}}   \prod_{{\color{violet}i}=0}^{{\color{blue}j}-1} {\color{red}v}_{3{\color{violet}i}+1}^{{\color{teal}s}_{3{\color{blue}j}+2}-{\color{teal}s}_{3{\color{violet}i}+1}} {\color{red}v}_{3{\color{violet}i}+2}^{{\color{teal}s}_{3{\color{blue}j}+2}-{\color{teal}s}_{3{\color{violet}i}+2}} {\color{red}v}_{3{\color{violet}i}+3}^{{\color{teal}s}_{3{\color{blue}j}+2}-{\color{teal}s}_{3{\color{violet}i}+3}}  \times {\color{red}v}_{3{\color{blue}j}+1}^{{\color{teal}s}_{3{\color{blue}j}+2}-{\color{teal}s}_{3{\color{blue}j}+1}}\\
&={\color{cyan}\xi}_3  {\color{red}v}_{3{\color{blue}j}+1}   \prod_{{\color{violet}i}=0}^{{\color{blue}j}-1} {\color{red}v}_{3{\color{violet}i}+1}^{({\color{blue}j}-{\color{violet}i}){\color{red}m}+1} {\color{red}v}_{3{\color{violet}i}+2}^{ ({\color{blue}j}-{\color{violet}i}){\color{red}m} } {\color{red}v}_{3{\color{violet}i}+3}^{({\color{blue}j}-{\color{violet}i}){\color{red}m}-1}\\
&={\color{cyan}\xi}_3  {\color{red}v}_{3{\color{blue}j}+1} 
\prod_{{\color{violet}i}=0}^{{\color{blue}j}-1} \left({\color{red}v}_{3{\color{violet}i}+1} {\color{red}v}_{3{\color{violet}i}+2} {\color{red}v}_{3{\color{violet}i}+3} \right)^{({\color{blue}j}-{\color{violet}i}){\color{red}m}+1}
\prod_{{\color{violet}i}=0}^{{\color{blue}j}-1}  {\color{red}v}_{3{\color{violet}i}+2}^{ -1 } {\color{red}v}_{3{\color{violet}i}+3}^{-2}\\
&={\color{cyan}\xi}_3 {\color{red}v}_{3{\color{blue}j}+1} 
\prod_{{\color{violet}i}=0}^{{\color{blue}j}-1} \left( {\color{cyan}t}^{\color{red}m}\right)^{({\color{blue}j}-{\color{violet}i}){\color{red}m}+1}
/ \prod_{{\color{violet}i}=0}^{{\color{blue}j}-1}   \frac{ - {\color{cyan}t}^{2{\color{red}m}}[{\color{violet}i}+1]^2_{{\color{cyan}t}^{\color{red}m}}   }{ [{\color{violet}i}+2]^2_{{\color{cyan}t}^{\color{red}m}}}\\
&=(-1)^{{\color{red}m}-1+({\color{blue}j}+1){\color{red}m}({\color{red}m}-1)/2} {\color{teal}R}({\color{blue}j}) {\color{cyan}t}^{{\color{red}m}{\color{blue}j}({\color{red}m}{\color{blue}j}+{\color{red}m}-2)/2},
\end{align*}
and
\begin{align*}
{\color{violet}H}_{{\color{teal}s}_{3{\color{blue}j}+3}}({\color{blue}F}({\color{magenta}q}))
&=(-1)^{{\color{red}\epsilon}_{3{\color{blue}j}+3}}  {\color{red}v}_1^{{\color{teal}s}_{3{\color{blue}j}+3}-{\color{teal}s}_1} {\color{red}v}_2^{{\color{teal}s}_{3{\color{blue}j}+3}-{\color{teal}s}_2} \cdots {\color{red}v}_{3{\color{blue}j}+2}^{{\color{teal}s}_{3{\color{blue}j}+3}-{\color{teal}s}_{3{\color{blue}j}+2}},\\
&=(-1)^{{\color{red}\epsilon}_{3{\color{blue}j}+3}}   \prod_{{\color{violet}i}=0}^{{\color{blue}j}-1} {\color{red}v}_{3{\color{violet}i}+1}^{{\color{teal}s}_{3{\color{blue}j}+3}-{\color{teal}s}_{3{\color{violet}i}+1}} {\color{red}v}_{3{\color{violet}i}+2}^{{\color{teal}s}_{3{\color{blue}j}+3}-{\color{teal}s}_{3{\color{violet}i}+2}} {\color{red}v}_{3{\color{violet}i}+3}^{{\color{teal}s}_{3{\color{blue}j}+3}-{\color{teal}s}_{3{\color{violet}i}+3}}   \\
	& \quad \times 
{\color{red}v}_{3{\color{blue}j}+1}^{{\color{teal}s}_{3{\color{blue}j}+3}-{\color{teal}s}_{3{\color{blue}j}+1}}
{\color{red}v}_{3{\color{blue}j}+2}^{{\color{teal}s}_{3{\color{blue}j}+3}-{\color{teal}s}_{3{\color{blue}j}+2}}\\
&={\color{cyan}\xi}_3 {\color{red}v}_{3{\color{blue}j}+1}^2 {\color{red}v}_{3{\color{blue}j}+2} 
\prod_{{\color{violet}i}=0}^{{\color{blue}j}-1} \left({\color{red}v}_{3{\color{violet}i}+1} {\color{red}v}_{3{\color{violet}i}+2} {\color{red}v}_{3{\color{violet}i}+3} \right)^{({\color{blue}j}-{\color{violet}i}){\color{red}m}+2}
\prod_{{\color{violet}i}=0}^{{\color{blue}j}-1}  {\color{red}v}_{3{\color{violet}i}+2}^{ -1 } {\color{red}v}_{3{\color{violet}i}+3}^{-2}\\
&={\color{cyan}\xi}_3  \frac{[{\color{blue}j}+2]^2_{{\color{cyan}t}^{\color{red}m}}}{[{\color{blue}j}+1]^2_{{\color{cyan}t}^{\color{red}m}}}
\prod_{{\color{violet}i}=0}^{{\color{blue}j}-1} \left( {\color{cyan}t}^{\color{red}m}\right)^{({\color{blue}j}-{\color{violet}i}){\color{red}m}+2}
/ \prod_{{\color{violet}i}=0}^{{\color{blue}j}-1}   \frac{ - {\color{cyan}t}^{2{\color{red}m}}[{\color{violet}i}+1]^2_{{\color{cyan}t}^{\color{red}m}}   }{ [{\color{violet}i}+2]^2_{{\color{cyan}t}^{\color{red}m}}}\\
&=(-1)^{({\color{blue}j}+1){\color{red}m}({\color{red}m}-1)/2}  {\color{cyan}t}^{{\color{red}m}^2 {\color{blue}j}({\color{blue}j}+1)/2 }   [{\color{blue}j}+2]^2_{{\color{cyan}t}^{\color{red}m}}.\qedhere
\end{align*}
\end{proof}


\section{Basic transformations for quadratic power series}\label{sec:quadra} 

\begin{Lemma}\label{lem:ABC:operator}
Let ${\color{teal}A},{\color{magenta}B},{\color{cyan}C},{\color{purple}U}$ be polynomials.
Suppose ${\color{blue}F}$ is a quadratic power series satisfying
\begin{equation*}
0={\color{teal}A} + {\color{magenta}B} {\color{blue}F} + {\color{cyan}C} {\color{blue}F}^2.
\end{equation*}
Then the series ${\color{purple}U}{\color{blue}F}$, ${\color{blue}F}+{\color{purple}U}$, and ${\color{blue}F}^{\color{purple}n}$ are also quadratic power series, and they satisfy
\begin{align*}
0&={\color{teal}A}{\color{purple}U}^2 + {\color{magenta}B}{\color{purple}U} ({\color{purple}U}{\color{blue}F}) + {\color{cyan}C}  ({\color{purple}U}{\color{blue}F})^2,\\
0&=({\color{teal}A} -{\color{magenta}B}{\color{purple}U} +{\color{cyan}C}{\color{purple}U}^2)  + ({\color{magenta}B}   - 2{\color{cyan}C}{\color{purple}U})({\color{blue}F}+{\color{purple}U}) + {\color{cyan}C} ({\color{blue}F}+{\color{purple}U})^2,\\
0&={\color{teal}A}^{\color{purple}n} + (-1)^{{\color{purple}n}+1} {\color{cyan}L}_{\color{purple}n} ({\color{magenta}B}, -{\color{teal}A}{\color{cyan}C}) {\color{blue}F}^{\color{purple}n} + {\color{cyan}C}^{\color{purple}n} {\color{blue}F}^{2 {\color{purple}n}}.
\end{align*}
\end{Lemma}
\begin{proof}
The initial two situations, ${\color{purple}U}{\color{blue}F}$ and ${\color{blue}F}+{\color{purple}U}$, are straightforward, so we focus on the case ${\color{red}G}={\color{blue}F}^{\color{purple}n}$.

Begin with the special case ${\color{teal}A}={\color{cyan}C}=1$, and assume
\begin{align}
0&=1 + {\color{magenta}B} {\color{blue}F} +  {\color{blue}F}^2,\label{eq:quadraB}\\
0&=1 + {\color{magenta}B}_{\color{purple}n} {\color{blue}F}^{\color{purple}n} +  {\color{blue}F}^{2 {\color{purple}n}}.\label{eq:quadraBk}
\end{align}
Clearly, ${\color{magenta}B}_0=-2$ and ${\color{magenta}B}_1={\color{magenta}B}$. From \eqref{eq:quadraB} and \eqref{eq:quadraBk} we derive
\begin{align*}
{\color{magenta}B}{\color{magenta}B}_{\color{purple}n} &=\left(\frac{1}{{\color{blue}F}}  +  {\color{blue}F}\right)\left(\frac{1}{{\color{blue}F}^{\color{purple}n}}  +  {\color{blue}F}^{ {\color{purple}n}}\right)\\
	&=\left(\frac{1}{{\color{blue}F}^{{\color{purple}n}+1}} +  {\color{blue}F}^{{\color{purple}n}+1}\right) +  \left({\color{blue}F}^{{\color{purple}n}-1} + \frac{1}{{\color{blue}F}^{{\color{purple}n}-1}}\right)\\
&=-{\color{magenta}B}_{{\color{purple}n}+1} -{\color{magenta}B}_{{\color{purple}n}-1}.
\end{align*}
Comparing this with \eqref{Lucas:rec}, we obtain
\begin{align}
{\color{magenta}B}_{\color{purple}n} &=(-1)^{{\color{purple}n}+1} {\color{cyan}L}_{\color{purple}n}({\color{magenta}B},-1)\nonumber\\
&=(-1)^{{\color{purple}n}+1}\sum_{{\color{olive}k}=0}^{\lfloor {\color{purple}n}/2\rfloor} 
	\binom{{\color{purple}n}-{\color{olive}k}}{{\color{olive}k}} \frac{{\color{purple}n}}{{\color{purple}n}-{\color{olive}k}} (-1)^{\color{olive}k} {\color{magenta}B}^{{\color{purple}n}-2{\color{olive}k}}.\label{eq:specBn}
\end{align}

Now consider the general case, where
\begin{align}
0&={\color{teal}A} + {\color{magenta}B} {\color{blue}F} + {\color{cyan}C} {\color{blue}F}^2,\nonumber\\
0&={\color{teal}A}_{\color{purple}n} + {\color{magenta}B}_{\color{purple}n} {\color{blue}F}^{\color{purple}n} + {\color{cyan}C}_{\color{purple}n} {\color{blue}F}^{2 {\color{purple}n}}.\label{eq:generalBn}
\end{align}
Set ${\color{blue}F} = \sqrt{{\color{teal}A}/{\color{cyan}C}}\, {\color{red}G}$. Then the first relation becomes
\begin{align*}
0&={\color{teal}A} + {\color{magenta}B} \sqrt{\frac{{\color{teal}A}}{{\color{cyan}C}}}\, {\color{red}G} + {\color{cyan}C}   \left(\sqrt{\frac{{\color{teal}A}}{{\color{cyan}C}}}\, {\color{red}G}\right)^2,\\
0&=1 + \frac{{\color{magenta}B}}{\sqrt{{\color{teal}A}{\color{cyan}C}}}\, {\color{red}G} +     {\color{red}G}^2.
\end{align*}
By \eqref{eq:specBn} it follows that
\begin{equation}\label{eq:Gn}
0=1 + {\bar {\color{magenta}B}}_{\color{purple}n} {\color{red}G}^{\color{purple}n} +     {\color{red}G}^{2{\color{purple}n}}, 
\end{equation}
where ${\bar {\color{magenta}B}}_0=-2$ and
\begin{equation*}
{\bar {\color{magenta}B}}_{\color{purple}n} = (-1)^{{\color{purple}n}+1}\sum_{{\color{olive}k}=0}^{\lfloor {\color{purple}n}/2\rfloor} 
\binom{{\color{purple}n}-{\color{olive}k}}{{\color{olive}k}} \frac{{\color{purple}n}}{{\color{purple}n}-{\color{olive}k}} (-1)^{\color{olive}k} \left(\frac{{\color{magenta}B}}{\sqrt{{\color{teal}A}{\color{cyan}C}}}\right)^{{\color{purple}n}-2{\color{olive}k}}.
\end{equation*}
On the other hand, from \eqref{eq:Gn} we get
\begin{align*}
0&=1 + {\bar {\color{magenta}B}}_{\color{purple}n} \left(\sqrt{\frac{{\color{cyan}C}}{{\color{teal}A}}}\right)^{\!{\color{purple}n}} {\color{blue}F}^{\color{purple}n} +    \left(\frac{{\color{cyan}C}}{{\color{teal}A}}\right)^{\!{\color{purple}n}} {\color{blue}F}^{2{\color{purple}n}}\\
&={\color{teal}A}^{\color{purple}n} + {\bar {\color{magenta}B}}_{\color{purple}n} (\sqrt{{\color{teal}A}{\color{cyan}C}})^{\color{purple}n}  {\color{blue}F}^{\color{purple}n} +    {\color{cyan}C}^{\color{purple}n} {\color{blue}F}^{2{\color{purple}n}}.
\end{align*}
Comparing this with \eqref{eq:generalBn}, we identify ${\color{teal}A}_{\color{purple}n}={\color{teal}A}^{\color{purple}n}$, ${\color{cyan}C}_{\color{purple}n}={\color{cyan}C}^{\color{purple}n}$, and
\begin{align*}
{\color{magenta}B}_{\color{purple}n} 
	&={\bar {\color{magenta}B}}_{\color{purple}n} (\sqrt{{\color{teal}A}{\color{cyan}C}})^{\color{purple}n}\\
&=(-1)^{{\color{purple}n}+1}\sum_{{\color{olive}k}=0}^{\lfloor {\color{purple}n}/2\rfloor} 
\binom{{\color{purple}n}-{\color{olive}k}}{{\color{olive}k}} \frac{{\color{purple}n}}{{\color{purple}n}-{\color{olive}k}} (-1)^{\color{olive}k} \left(\frac{{\color{magenta}B}}{\sqrt{{\color{teal}A}{\color{cyan}C}}}\right)^{{\color{purple}n}-2{\color{olive}k}}
 (\sqrt{{\color{teal}A}{\color{cyan}C}})^{\color{purple}n}\\
&=(-1)^{{\color{purple}n}+1}\sum_{{\color{olive}k}=0}^{\lfloor {\color{purple}n}/2\rfloor} 
\binom{{\color{purple}n}-{\color{olive}k}}{{\color{olive}k}} \frac{{\color{purple}n}}{{\color{purple}n}-{\color{olive}k}} (-1)^{\color{olive}k} 
{\color{magenta}B}^{{\color{purple}n}-2{\color{olive}k}}({\color{teal}A}{\color{cyan}C})^{{\color{olive}k}}\\
&=(-1)^{{\color{purple}n}+1}{\color{cyan}L}_{\color{purple}n} ({\color{magenta}B}, -{\color{teal}A}{\color{cyan}C}).\qedhere
\end{align*}
\end{proof}

\begin{Corollary}\label{cor:F}
The power series
\begin{equation*}
{\color{blue}F}({\color{magenta}q})=\frac{({\color{blue}\gamma}({\color{cyan}t},{\color{magenta}q})-1)^{\color{red}m}}{{\color{magenta}q}^{{\color{red}m}_0}}
\end{equation*}
satisfies the following quadartic equation
\begin{equation*}
0= - {\color{magenta}q}^{{\color{red}m}-{\color{red}m}_0} +     {\color{magenta}\beta}({\color{red}m};{\color{cyan}t},{\color{magenta}q})  {\color{blue}F}  - {\color{cyan}t}^{\color{red}m}  {\color{magenta}q}^{{\color{red}m}+{\color{red}m}_0}  {\color{blue}F}^{2}.
\end{equation*}
\end{Corollary}
\begin{proof}
	Invoking Lemma \ref{lem:ABC:operator} in the case ${\color{blue}F}+{\color{purple}U}$ and applying it to \eqref{eq:Narayana}, we obtain
\begin{equation*}
	-{\color{magenta}q}+ (1 - {\color{magenta}q} - {\color{cyan}t}{\color{magenta}q})\,({\color{blue}\gamma}({\color{cyan}t},{\color{magenta}q})-1) - {\color{cyan}t}{\color{magenta}q}\,({\color{blue}\gamma}({\color{cyan}t},{\color{magenta}q})-1)^2 = 0.
\end{equation*}
	Next, applying Lemma \ref{lem:ABC:operator} once more, now in the case ${\color{blue}F}^{\color{purple}n}$, and using Lemma \ref{lem:beta:Lm}, we deduce
\begin{equation*}
	-{\color{magenta}q}^{\color{red}m}+ {\color{magenta}\beta}({\color{red}m};{\color{cyan}t},{\color{magenta}q}) ({\color{blue}\gamma}({\color{cyan}t},{\color{magenta}q})-1)^{\color{red}m} -  ({\color{cyan}t}{\color{magenta}q})^{\color{red}m}\,({\color{blue}\gamma}({\color{cyan}t},{\color{magenta}q})-1)^{2{\color{red}m}} = 0.
\end{equation*}
The statement of the corollary follows by one further application of the same lemma, this time in the case ${\color{blue}F}{\color{purple}U}$.
\end{proof}


\section{Hankel continued fraction of quadratic series} 
To derive the Hankel continued fraction of a quadratic series, we recall the main idea presented in \cite{Han2016Adv, Han2025Pedon}. Suppose ${\color{blue}F}({\color{magenta}q})$ is a power series that satisfies the quadratic equation
$${\color{teal}A}({\color{magenta}q})+{\color{magenta}B}({\color{magenta}q}){\color{blue}F}({\color{magenta}q})+ {\color{cyan}C}({\color{magenta}q}){\color{blue}F}({\color{magenta}q})^2 = 0.$$
Then ${\color{blue}F}({\color{magenta}q})$ can be expressed in the form
$$
{\color{blue}F}({\color{magenta}q}) = \frac{-{\color{purple}a} {\color{magenta}q}^{\color{olive}k}}{{\color{cyan}D}({\color{magenta}q}) - {\color{magenta}q}^{{\color{olive}k}+{\color{violet}\delta}} {\color{red}G}({\color{magenta}q})},
$$
where ${\color{red}G}({\color{magenta}q})$ is another power series that satisfies a transformed quadratic equation
$${\color{teal}A}^*({\color{magenta}q})+{\color{magenta}B}^*({\color{magenta}q}){\color{red}G}({\color{magenta}q})+ {\color{cyan}C}^*({\color{magenta}q}){\color{red}G}({\color{magenta}q})^2 = 0.$$
Moreover, the quantities ${\color{teal}A}^*, {\color{magenta}B}^*, {\color{cyan}C}^*, {\color{olive}k}, {\color{purple}a},$ and ${\color{cyan}D}$ can be computed explicitly by means of the following algorithm.

\medskip

{\bf Algorithm} \hbox{\tt [NextABC]} (for ${\color{violet}\delta}=2$)

\smallskip\noindent

Prototype: $({\color{teal}A}^*, {\color{magenta}B}^*, {\color{cyan}C}^*; {\color{olive}k}, {\color{purple}a}, {\color{cyan}D})=\hbox{\tt NextABC}({\color{teal}A},{\color{magenta}B},{\color{cyan}C})$

\smallskip\noindent

Input: ${\color{teal}A}({\color{magenta}q}), {\color{magenta}B}({\color{magenta}q}), {\color{cyan}C}({\color{magenta}q})\in \QQ[{\color{magenta}q}]$ three polynomials such that 
${\color{magenta}B}(0)=1,$ ${\color{cyan}C}(0)=0$, ${\color{cyan}C}({\color{magenta}q})\not=0, {\color{teal}A}({\color{magenta}q})\not=0$;

\smallskip\noindent

Output: ${\color{teal}A}^*({\color{magenta}q}), {\color{magenta}B}^*({\color{magenta}q}), {\color{cyan}C}^*({\color{magenta}q})\in\QQ[{\color{magenta}q}]$, ${\color{olive}k}\in \NN^+$, ${\color{purple}a}\not=0 \in\QQ$,
${\color{cyan}D}({\color{magenta}q})\in\QQ[{\color{magenta}q}]$ a polynomial of degree less than or equal to ${\color{olive}k}+1$ such that ${\color{cyan}D}(0)=1$.

\smallskip\noindent

Step 1 [Define ${\color{olive}k}, {\color{purple}a}$]. Since ${\color{teal}A}({\color{magenta}q})\not=0$, let ${\color{teal}A}({\color{magenta}q})={\color{purple}a}{\color{magenta}q}^{\color{olive}k} + {\color{magenta}O}({\color{magenta}q}^{{\color{olive}k}+1})$ with ${\color{purple}a}\not=0$.

\smallskip\noindent

Step 2 [Define ${\color{cyan}D}$].  Define ${\color{cyan}D}({\color{magenta}q})$  by
\begin{equation}\label{defD2}\frac{{\color{purple}a}{\color{magenta}q}^{\color{olive}k}{\color{magenta}B}}{{\color{teal}A}}  - \frac{{\color{purple}a}{\color{magenta}q}^{\color{olive}k}  {\color{cyan}C}}{ {\color{magenta}B}} ={\color{cyan}D}({\color{magenta}q})  +  {\color{magenta}O}({\color{magenta}q}^{{\color{olive}k}+2 }),  
\end{equation}
where ${\color{cyan}D}({\color{magenta}q})$ is a polynomial of degree less than or equal to ${\color{olive}k}+1$ such that ${\color{cyan}D}(0)=1$.

\smallskip\noindent

Step 3. [Define ${\color{teal}A}^*, {\color{magenta}B}^*, {\color{cyan}C}^*$]. Let
\begin{align}
{{\color{teal}A}^*}({\color{magenta}q}) &=\bigl(-{\color{cyan}D}^2{\color{teal}A}/{\color{purple}a}+{\color{magenta}B}{\color{cyan}D} {\color{magenta}q}^{\color{olive}k}-{\color{cyan}C}{\color{purple}a}{\color{magenta}q}^{2{\color{olive}k}}\bigr)/{\color{magenta}q}^{2{\color{olive}k}+2};\label{eqAA}\\
	{{\color{magenta}B}^*} ({\color{magenta}q}) &=2{\color{teal}A}{\color{cyan}D}/({\color{purple}a}{\color{magenta}q}^{{\color{olive}k}}) - {\color{magenta}B} ;\label{eqBB}\\
	{{\color{cyan}C}^*} ({\color{magenta}q}) &=-{\color{teal}A}{\color{magenta}q}^{2}/{\color{purple}a}.\label{eqCC}
\end{align}
{\it Remark}. In Step 2, if ${\color{cyan}C} = {\color{magenta}q}^2 {\color{cyan}C}'$, then
\begin{equation*}
\frac{{\color{purple}a} {\color{magenta}q}^{\color{olive}k} {\color{magenta}B}}{{\color{teal}A}} = {\color{cyan}D}({\color{magenta}q}) + {\color{magenta}O}({\color{magenta}q}^{{\color{olive}k}+2}).
\end{equation*}
Similarly, if ${\color{teal}A} = {\color{purple}a} {\color{magenta}q}^{\color{olive}k}$, then
\begin{equation*}
{\color{magenta}B} = {\color{cyan}D}({\color{magenta}q}) + {\color{magenta}O}({\color{magenta}q}^{{\color{olive}k}+2}).
\end{equation*}

By repeatedly applying Algorithm \hbox{\tt NextABC}, we generate a sequence of six-tuples  
\begin{equation}\label{sixtuples}
({\color{teal}A}_{{\color{purple}n}+1}, {\color{magenta}B}_{{\color{purple}n}+1}, {\color{cyan}C}_{{\color{purple}n}+1}; {\color{olive}k}_{{\color{purple}n}}, {\color{purple}a}_{{\color{purple}n}}, {\color{cyan}D}_{{\color{purple}n}}),
\end{equation}
satisfying, for each ${\color{purple}n}$,
\begin{equation}\label{ABCcheck}
({\color{teal}A}_{{\color{purple}n}+1}, {\color{magenta}B}_{{\color{purple}n}+1}, {\color{cyan}C}_{{\color{purple}n}+1}; {\color{olive}k}_{{\color{purple}n}}, {\color{purple}a}_{{\color{purple}n}}, {\color{cyan}D}_{{\color{purple}n}}) = \hbox{\tt NextABC}({\color{teal}A}_{{\color{purple}n}}, {\color{magenta}B}_{{\color{purple}n}}, {\color{cyan}C}_{{\color{purple}n}}).
\end{equation}
This yields an ${\color{violet}H}$-fraction expansion of ${\color{blue}F}({\color{magenta}q})$ of the form
$$
{\color{blue}F}({\color{magenta}q}) = \mathcal{{\color{violet}H}}(({\color{olive}k}_{{\color{blue}j}}), (-{\color{purple}a}_{{\color{blue}j}}), ({\color{red}u}_{{\color{blue}j}})),
$$
where ${\color{red}u}_{\color{blue}j}$ is determined by the relation $1+{\color{red}u}_{{\color{blue}j}}({\color{magenta}q}){\color{magenta}q} = {\color{cyan}D}_{{\color{blue}j}-1}$.

\medskip

We will employ this method to establish the ${\color{violet}H}$-fraction representations stated in Section \ref{sec:hfrac}. To do this, we first conjecture the sequence \eqref{sixtuples} from computational data obtained experimentally, and then confirm its validity by verifying the relations \eqref{ABCcheck}. In other words, we must check the three steps of the algorithm. 


\section{Proof of Lemma \ref{lem:HFrac:Q2}}\label{sec:Q2} 
For ${\color{red}m}\geq 2$, let
\begin{equation*}
{\color{blue}F}({\color{magenta}q})=\frac{({\color{blue}\gamma}({\color{cyan}t},{\color{magenta}q})-1)^{\color{red}m}}{{\color{magenta}q}^2}.
\end{equation*}
By Corollary \ref{cor:F}, we have
\begin{equation*}
0= - {\color{magenta}q}^{{\color{red}m}-2} +     {\color{magenta}\beta}({\color{red}m},{\color{cyan}t},{\color{magenta}q})  {\color{blue}F}  - {\color{cyan}t}^{\color{red}m}  {\color{magenta}q}^{{\color{red}m}+2}  {\color{blue}F}^{2}.
\end{equation*}
Thus, we run Algorithm \hbox{\tt NextABC} starting from the initialization
$${\color{teal}A}_0=-{\color{magenta}q}^{{\color{red}m}-2},\quad {\color{magenta}B}_0={\color{magenta}\beta}({\color{red}m};{\color{cyan}t},{\color{magenta}q}),\quad {\color{cyan}C}_0=-{\color{cyan}t}^{\color{red}m}{\color{magenta}q}^{{\color{red}m}+2}.$$

\begin{Lemma}\label{lem:ABCq2}
If we use ${\color{teal}A}_{{\color{purple}n}}, {\color{magenta}B}_{{\color{purple}n}}, {\color{cyan}C}_{{\color{purple}n}}$ as the input to Algorithm \hbox{\tt NextABC},
then the outputs are
$$ {\color{teal}A}_{{\color{purple}n}+1}, {\color{magenta}B}_{{\color{purple}n}+1}, {\color{cyan}C}_{{\color{purple}n}+1}, {\color{olive}k}_{{\color{purple}n}}, {\color{purple}a}_{{\color{purple}n}}, {\color{cyan}D}_{{\color{purple}n}}, $$
where
${\color{teal}A}_{{\color{purple}n}}, {\color{magenta}B}_{{\color{purple}n}}, {\color{cyan}C}_{{\color{purple}n}}, {\color{olive}k}_{{\color{purple}n}}, {\color{purple}a}_{{\color{purple}n}}, {\color{cyan}D}_{{\color{purple}n}}$ for ${\color{purple}n} \geq 0$ are defined as follows:
\begin{align*}
	{\color{teal}A}_{0}&=-{\color{magenta}q}^{{\color{red}m}-2}, \qquad
	{\color{teal}A}_{2{\color{blue}j}}=\frac{ (-{\color{cyan}t})^{\color{red}m} {\color{magenta}q}^{{\color{red}m}-2} {\color{olive}J}_0}{{\color{olive}J}_1},\\
	{\color{teal}A}_{2{\color{blue}j}+1}&=\frac{(-1)^{\color{red}m}{\color{olive}J}_2}{{\color{olive}J}_1}  {\color{magenta}\beta}({\color{red}m};{\color{cyan}t},{\color{magenta}q})
 + 
(\frac{1-2{\color{olive}J}_2}{{\color{olive}J}_1}-\frac{{\color{cyan}t}^{{\color{red}m}({\color{blue}j}+1)}{\color{olive}J}_2}{{\color{olive}J}_1^2} ) {\color{magenta}q}^{\color{red}m},\\
	{\color{magenta}B}_{2{\color{blue}j}}&={\color{magenta}\beta}({\color{red}m};{\color{cyan}t},{\color{magenta}q}) + 2(\frac{1}{{\color{olive}J}_1} - 1) (-{\color{magenta}q})^{\color{red}m},\\
	{\color{magenta}B}_{2{\color{blue}j}+1}&={\color{magenta}\beta}({\color{red}m};{\color{cyan}t},{\color{magenta}q}) - 2(\frac{{\color{cyan}t}^{{\color{red}m}({\color{blue}j}+1)}}{{\color{olive}J}_1} + 1) (-{\color{magenta}q})^{\color{red}m},\\
	{\color{cyan}C}_{0}&=-{\color{cyan}t}^{\color{red}m} {\color{magenta}q}^{{\color{red}m}+2}, \qquad {\color{cyan}C}_{2{\color{blue}j}+1}=-{\color{magenta}q}^{\color{red}m},\\
	{\color{cyan}C}_{2{\color{blue}j}}&=-{\color{magenta}\beta}({\color{red}m};{\color{cyan}t},{\color{magenta}q}) {\color{magenta}q}^2 - ( \frac{1}{{\color{olive}J}_1}  - \frac{{\color{cyan}t}^{{\color{red}m}{\color{blue}j}}}{{\color{olive}J}_0}-2 ) (-{\color{magenta}q})^{{\color{red}m}+2}\\
	{\color{olive}k}_{2{\color{blue}j}}&={\color{red}m}-2, \qquad {\color{olive}k}_{2{\color{blue}j}+1}=0,\\
	{\color{purple}a}_{0}&=-1, \qquad
	{\color{purple}a}_{2{\color{blue}j}}=\frac{(-{\color{cyan}t})^{\color{red}m}  {\color{olive}J}_0}{ {\color{olive}J}_1},\qquad
	{\color{purple}a}_{2{\color{blue}j}+1}=\frac{(-1)^{\color{red}m}  {\color{olive}J}_2}{{\color{olive}J}_1},\\
	{\color{cyan}D}_{2{\color{blue}j}}&={\color{magenta}\beta}({\color{red}m};{\color{cyan}t},{\color{magenta}q}) - (-{\color{magenta}q})^{\color{red}m}(1+{\color{cyan}t}^{\color{red}m}), \qquad
	{\color{cyan}D}_{2{\color{blue}j}+1}=1,
\end{align*}
where, for brevity, we denote
\begin{align*}
{\color{olive}J}_0&=[{\color{blue}j}]_{{\color{cyan}t}^{\color{red}m}},\qquad
{\color{olive}J}_1=[{\color{blue}j}+1]_{{\color{cyan}t}^{\color{red}m}},\qquad
{\color{olive}J}_2=[{\color{blue}j}+2]_{{\color{cyan}t}^{\color{red}m}}.
\end{align*}

\end{Lemma}

\begin{proof}
The values of ${\color{olive}k}_{\color{blue}j}$ and ${\color{purple}a}_{\color{blue}j}$ are straightforward to determine. For ${\color{cyan}D}_{\color{blue}j}$, we distinguish two cases.

In the even case, we have ${\color{olive}k}_{2{\color{blue}j}} = {\color{red}m} - 2$. Since ${\color{cyan}C}_{\color{purple}n} = {\color{magenta}O}({\color{magenta}q}^2)$, it follows that
\begin{equation*}
	\frac{{\color{purple}a}_{2{\color{blue}j}} {\color{magenta}q}^{\color{olive}k} {\color{magenta}B}_{2{\color{blue}j}}}{{\color{teal}A}_{2{\color{blue}j}}} = {\color{cyan}D}_{2{\color{blue}j}}({\color{magenta}q}) + {\color{magenta}O}({\color{magenta}q}^{{\color{red}m}}),
\end{equation*}
which implies
\begin{equation*}
	{\color{magenta}B}_{2{\color{blue}j}} = {\color{cyan}D}_{2{\color{blue}j}}({\color{magenta}q}) + {\color{magenta}O}({\color{magenta}q}^{{\color{red}m}}).
\end{equation*}
Thus, ${\color{cyan}D}_{2{\color{blue}j}}$ is simply the polynomial obtained from ${\color{magenta}B}_{2{\color{blue}j}}$ by truncating its ${\color{magenta}q}$-expansion at order ${\color{magenta}q}^{{\color{red}m}-1}$.

In the odd case, we have ${\color{olive}k}_{2{\color{blue}j}+1} = 0$, and
\begin{equation*}
	\frac{{\color{purple}a}_{2{\color{blue}j}+1} {\color{magenta}B}_{2{\color{blue}j}+1}}{{\color{teal}A}_{2{\color{blue}j}+1}} = {\color{cyan}D}_{2{\color{blue}j}+1} + {\color{magenta}O}({\color{magenta}q}^{2}).
\end{equation*}
Hence, ${\color{cyan}D}_{2{\color{blue}j}+1} = 1$.

The remaining three relations, \eqref{eqAA}, \eqref{eqBB}, and \eqref{eqCC}, are verified using a computer algebra system.
\end{proof}

\section{Proof of Lemma \ref{lem:HFrac:Q3}}\label{sec:Q3} 
For ${\color{red}m}\geq 3$, let
\begin{equation*}
{\color{blue}F}({\color{magenta}q})=\frac{({\color{blue}\gamma}({\color{cyan}t},{\color{magenta}q})-1)^{\color{red}m}}{{\color{magenta}q}^3}.
\end{equation*}
By Corollary \ref{cor:F}, we have
\begin{equation*}
0= - {\color{magenta}q}^{{\color{red}m}-3} +     {\color{magenta}\beta}({\color{red}m},{\color{cyan}t},{\color{magenta}q})  {\color{blue}F}  - {\color{cyan}t}^{\color{red}m}  {\color{magenta}q}^{{\color{red}m}+3}  {\color{blue}F}^{2}
\end{equation*}
So we run Algorithm \hbox{\tt NextABC} starting from the initial values
$${\color{teal}A}_0=-{\color{magenta}q}^{{\color{red}m}-3},\quad {\color{magenta}B}_0={\color{magenta}\beta}({\color{red}m};{\color{cyan}t},{\color{magenta}q}),\quad {\color{cyan}C}_0=-{\color{cyan}t}^{\color{red}m}{\color{magenta}q}^{{\color{red}m}+3}.$$

\begin{Lemma}\label{lem:ABCq3}
If we input ${\color{teal}A}_{{\color{purple}n}}, {\color{magenta}B}_{{\color{purple}n}}, {\color{cyan}C}_{{\color{purple}n}}$ into Algorithm \hbox{\tt NextABC}, then the algorithm produces
$$ {\color{teal}A}_{{\color{purple}n}+1}, {\color{magenta}B}_{{\color{purple}n}+1}, {\color{cyan}C}_{{\color{purple}n}+1}, {\color{olive}k}_{{\color{purple}n}}, {\color{purple}a}_{{\color{purple}n}}, {\color{cyan}D}_{{\color{purple}n}}, $$
where, for ${\color{purple}n} \geq 0$, the sequences ${\color{teal}A}_{{\color{purple}n}}, {\color{magenta}B}_{{\color{purple}n}}, {\color{cyan}C}_{{\color{purple}n}}, {\color{olive}k}_{{\color{purple}n}}, {\color{purple}a}_{{\color{purple}n}}, {\color{cyan}D}_{{\color{purple}n}}$ are defined as follows:
\begin{align*}
	{\color{teal}A}_{0}&=-{\color{magenta}q}^{{\color{red}m}-3}, \qquad
	{\color{teal}A}_{3{\color{blue}j}}=(-{\color{cyan}t})^{\color{red}m} {\color{magenta}q}^{{\color{red}m}-3} { {\color{teal}R}({\color{blue}j}-1) {\color{olive}J}_1^{-2} },\\
	{{\color{teal}A}_{3{\color{blue}j}+1}} 	&={\color{teal}\alpha}({\color{magenta}q})\left( {\color{red}u}({\color{blue}j})
	 + {(-1)^{\color{red}m} {\color{olive}J}_2 }{{\color{olive}J}_1^{-1}} {\color{magenta}q}
	\right)\\
& \quad - \left(
	(-1)^{{\color{red}m}+1}{\color{red}u}({\color{blue}j}) + {\color{blue}w}_3({\color{blue}j}) {\color{magenta}q} + {\color{cyan}t}^{{\color{red}m}{\color{blue}j}} {\color{magenta}q}^2
	\right) {{\color{cyan}t}^{\color{red}m} {\color{magenta}q}^{{\color{red}m}-1}}{ {\color{olive}J}_1^{-2}}{\color{teal}R}({\color{blue}j}-1),\\
	{\color{teal}A}_{3{\color{blue}j}+2} &=\frac{-{\color{red}v}({\color{blue}j})^2  {\color{cyan}C}_{3{\color{blue}j}+3}}{ {\color{magenta}q}^2},\\
	{\color{magenta}B}_{3{\color{blue}j}}&={\color{teal}\alpha}({\color{magenta}q}) + {\color{blue}w}_2({\color{blue}j}), \qquad
	{{\color{magenta}B}_{3{\color{blue}j}+1}}   ={\color{teal}\alpha}({\color{magenta}q})  -{\color{blue}w}_2({\color{blue}j}),\\
	{{\color{magenta}B}_{3{\color{blue}j}+2}} &=2  (1+{\color{magenta}q} {\color{red}v}({\color{blue}j}))  {\color{blue}w}_1({\color{blue}j}) -{\color{teal}\alpha}({\color{magenta}q})  + {\color{blue}w}_2({\color{blue}j}),\\
	{\color{cyan}C}_{0}&=-{\color{cyan}t}^{\color{red}m} {\color{magenta}q}^{{\color{red}m}+3},\\
	{\color{cyan}C}_{3{\color{blue}j}}&=-  {\color{teal}\alpha}({\color{magenta}q}) {\color{magenta}q}^2 (1+ {\color{red}v}({\color{blue}j}-1) {\color{magenta}q}) - {\color{red}u}({\color{blue}j}) {\color{magenta}q}^{{\color{red}m}+1}\\
& \quad + {\color{blue}w}_3({\color{blue}j}) (-{\color{magenta}q})^{{\color{red}m}+2} - {{\color{cyan}t}^{{\color{red}m}{\color{blue}j}}}{{\color{teal}R}({\color{blue}j}-1)^{-1}}  (-{\color{magenta}q})^{{\color{red}m}+3},\\
	{{\color{cyan}C}_{3{\color{blue}j}+1}}  &=-{\color{magenta}q}^{{\color{red}m}-1}, \qquad 
	{{\color{cyan}C}_{3{\color{blue}j}+2}} =-{\color{magenta}q}^{2}{\color{blue}w}_1({\color{blue}j}),\\
	{\color{olive}k}_{3{\color{blue}j}}&={\color{red}m}-3, \qquad {\color{olive}k}_{3{\color{blue}j}+1}= {\color{olive}k}_{3{\color{blue}j}+2}=0,\\
	{\color{purple}a}_{0}&=-1, \quad 
	{\color{purple}a}_{3{\color{blue}j}}=(-{\color{cyan}t})^{\color{red}m}   { {\color{teal}R}({\color{blue}j}-1) }{ {\color{olive}J}_1^{-2}  },\\
	{\color{purple}a}_{3{\color{blue}j}+1}&={\color{red}u}({\color{blue}j}), \qquad
	{\color{purple}a}_{3{\color{blue}j}+2} ={\color{red}v}({\color{blue}j})^2,\\
	{\color{cyan}D}_{3{\color{blue}j}}&={\color{teal}\alpha}({\color{magenta}q}), \qquad
	{\color{cyan}D}_{3{\color{blue}j}+1}=1+{\color{magenta}q} {\color{red}v}({\color{blue}j}), \qquad
	{\color{cyan}D}_{3{\color{blue}j}+2}=1- {\color{magenta}q}{\color{red}v}({\color{blue}j}) ,
\end{align*}
where, for brevity, we define
\begin{align*}
{\color{red}u}({\color{blue}j})&=\frac{(-1)^{{\color{red}m}+1}  {\color{teal}R}({\color{blue}j}) }{ {\color{olive}J}_1^2},\qquad
{\color{red}v}({\color{blue}j})=\frac{{\color{olive}J}_1 {\color{olive}J}_2}{{\color{teal}R}({\color{blue}j}) }, \qquad
	{\color{blue}w}_1({\color{blue}j})=\frac{{\color{teal}A}_{3{\color{blue}j}+1}}{{\color{purple}a}_{3{\color{blue}j}+1}},\\
{\color{blue}w}_2({\color{blue}j})&=(-{\color{magenta}q})^{{\color{red}m}-1} { {\color{olive}S}({\color{blue}j})   }{ {\color{olive}J}_1^{-2} } + (-{\color{magenta}q})^{\color{red}m} { (1+{\color{cyan}t}^{({\color{blue}j}+1){\color{red}m}}) }{{\color{olive}J}_1^{-1}},\\
{\color{blue}w}_3({\color{blue}j})&={\color{red}m}[{\color{red}m}]_{\color{cyan}t} {\color{red}v}({\color{blue}j}-1) - {(1+{\color{cyan}t}^{{\color{red}m}({\color{blue}j}+1)})}{{\color{olive}J}_1^{-1}}.
\end{align*}
\end{Lemma}

\begin{proof}
Turn $(3{\color{blue}j})$.  
	From ${\color{teal}A}_{0}$ we derive the formulas for ${\color{olive}k}_{0}$ and ${\color{purple}a}_{0}$.
	Using \eqref{defD2}, we obtain $ {\color{magenta}B}_{0} = {\color{cyan}D}_{0} + {\color{magenta}O}({\color{magenta}q}^{{\color{red}m}-1})$,
	which validates the expression of ${\color{cyan}D}_{0}$.
	For ${\color{blue}j}\geq 1$, ${\color{teal}A}_{3{\color{blue}j}}$ yields the expressions for ${\color{olive}k}_{3{\color{blue}j}}$ and ${\color{purple}a}_{3{\color{blue}j}}$.
	Again, by \eqref{defD2}, we have $ {\color{magenta}B}_{3{\color{blue}j}} = {\color{cyan}D}_{3{\color{blue}j}} + {\color{magenta}O}({\color{magenta}q}^{{\color{red}m}-1})$,
	which confirms the expression of ${\color{cyan}D}_{3{\color{blue}j}}$.

	Turn $(3{\color{blue}j}+1)$.  From the explicit forms of ${\color{teal}A}_{3{\color{blue}j}+1}$ and ${\color{magenta}B}_{3{\color{blue}j}+1}$, we obtain
	${\color{olive}k}_{3{\color{blue}j}+1}=0$ and ${\color{purple}a}_{3{\color{blue}j}+1}={\color{red}u}({\color{blue}j})$.
Define
\begin{equation*}
	{\color{blue}w}_1({\color{blue}j})=\frac{{\color{teal}A}_{3{\color{blue}j}+1}}{{\color{purple}a}_{3{\color{blue}j}+1}} = {\color{teal}\alpha}({\color{magenta}q})\bigl(1 - {\color{red}v}({\color{blue}j}){\color{magenta}q}\bigr) + {\color{magenta}O}({\color{magenta}q}^{{\color{red}m}-1}).
\end{equation*}
Consequently,
\begin{equation*}
	\frac{{\color{magenta}B}_{3{\color{blue}j}+1}}{{\color{blue}w}_1({\color{blue}j})}
	= \frac{1}{1 - {\color{red}v}({\color{blue}j}){\color{magenta}q}} + {\color{magenta}O}({\color{magenta}q}^{{\color{red}m}-1}).
\end{equation*}
	Since ${\color{red}m}\geq 3$, we have ${\color{cyan}C}_{3{\color{blue}j}+1}={\color{magenta}O}({\color{magenta}q}^2)$, which implies ${\color{cyan}D}_{3{\color{blue}j}+1}=1 + {\color{magenta}q} {\color{red}v}({\color{blue}j})$.

Turn $(3{\color{blue}j}+2)$.  
	We obtain ${\color{olive}k}_{3{\color{blue}j}+2}=0$ and ${\color{purple}a}_{3{\color{blue}j}+2} = {\color{red}v}({\color{blue}j})^2$.
	Because ${\color{cyan}C}_{3{\color{blue}j}+2} = {\color{magenta}O}({\color{magenta}q}^2)$, we get
\begin{align*}
	{\color{magenta}B}_{3{\color{blue}j}+2}({\color{magenta}q}) 
	&= 2{\color{cyan}D}_{3{\color{blue}j}+1}{\color{blue}w}_1({\color{blue}j}) - {\color{magenta}B}_{3{\color{blue}j}+1}\\
	&= 2(1+{\color{red}v}({\color{blue}j}){\color{magenta}q}){\color{teal}\alpha}({\color{magenta}q})(1-{\color{red}v}({\color{blue}j}){\color{magenta}q}) - {\color{teal}\alpha}({\color{magenta}q}) + {\color{magenta}O}({\color{magenta}q}^2)\\
	&= 1 - {\color{red}m}({\color{cyan}t}+1){\color{magenta}q} + {\color{magenta}O}({\color{magenta}q}^2).
\end{align*}
Thus,
\begin{equation*}
	\frac{{\color{purple}a}_{3{\color{blue}j}+2}{\color{magenta}B}_{3{\color{blue}j}+2}}{{\color{teal}A}_{3{\color{blue}j}+2}}  = 1 - {\color{red}v}({\color{blue}j}){\color{magenta}q} + {\color{magenta}O}({\color{magenta}q}^{2}).
\end{equation*}
	Hence ${\color{cyan}D}_{3{\color{blue}j}+2} = 1 - {\color{magenta}q} {\color{red}v}({\color{blue}j})$.
The remaining three relations \eqref{eqAA}, \eqref{eqBB}, and \eqref{eqCC} are verified using a computer algebra system.
\end{proof}


\section{Proof of Lemmas \ref{lem:HFrac:Q0} and \ref{lem:HFrac:Q1}} 
In this section, we obtain Lemma \ref{lem:HFrac:Q0} as a consequence of Lemma \ref{lem:HFrac:Q2}, and we establish Lemma \ref{lem:HFrac:Q1} by a direct calculation.

\begin{proof}[Proof of Lemma \ref{lem:HFrac:Q0}]
For ${\color{red}m}\geq 0$, define ${\color{blue}F}({\color{magenta}q})=({\color{blue}\gamma}({\color{cyan}t},{\color{magenta}q})-1)^{\color{red}m}$.
By Corollary \ref{cor:F}, we have
\begin{equation*}
0= - {\color{magenta}q}^{{\color{red}m}} + {\color{magenta}\beta}({\color{red}m};{\color{cyan}t},{\color{magenta}q})\,{\color{blue}F} - {\color{cyan}t}^{\color{red}m} {\color{magenta}q}^{{\color{red}m}} {\color{blue}F}^{2}.
\end{equation*}
Hence
\begin{equation*}
{\color{blue}F}
= \frac{{\color{magenta}q}^{\color{red}m}}{{\color{magenta}\beta}({\color{red}m};{\color{cyan}t},{\color{magenta}q}) - {\color{cyan}t}^{\color{red}m} {\color{magenta}q}^{\color{red}m} ({\color{blue}\gamma}({\color{cyan}t},{\color{magenta}q})-1)^{\color{red}m}}
= \frac{{\color{magenta}q}^{\color{red}m}}{{\color{magenta}\beta}({\color{red}m};{\color{cyan}t},{\color{magenta}q}) - {\color{cyan}t}^{\color{red}m} {\color{magenta}q}^{{\color{red}m}+2} \frac{({\color{blue}\gamma}({\color{cyan}t},{\color{magenta}q})-1)^{\color{red}m}}{{\color{magenta}q}^2}}.
\end{equation*}
Since the ${\color{violet}H}$-fraction of ${({\color{blue}\gamma}({\color{cyan}t},{\color{magenta}q})-1)^{\color{red}m}}/{{\color{magenta}q}^2}$ is determined in Lemma \ref{lem:HFrac:Q2}, this directly yields the ${\color{violet}H}$-fraction expansion of ${\color{blue}F}$.
\end{proof}

\begin{proof}[Proof of Lemma \ref{lem:HFrac:Q1}]
Set
$${\color{blue}F}= \frac{({\color{blue}\gamma}({\color{cyan}t},{\color{magenta}q})-1)^{\color{red}m}}{{\color{magenta}q}}.$$
By Corollary \ref{cor:F}, we obtain
$$
{\color{blue}F}=\frac{{\color{magenta}q}^{{\color{red}m}-1}}{{\color{magenta}\beta}({\color{red}m};{\color{cyan}t},{\color{magenta}q}) - {\color{cyan}t}^{\color{red}m} {\color{magenta}q}^{{\color{red}m}+1} {\color{blue}F}}.
$$
This identity provides the ${\color{violet}H}$-fraction representation of ${\color{blue}F}$.
\end{proof}

\section{Fibonacci polynomials and Lucas polynomials} 
In exploring the central topic of this paper, we obtained several identities involving the Fibonacci and Lucas polynomials. Although these results are not required for the final exposition, we still find it worthwhile to record them here.

Recall that the {\it Fibonacci polynomials} \cite{Cigler2018} are defined for 
${\color{purple}n}\geq 0$ by
\begin{equation}\label{Fibo:sum}
{\color{blue}F}_{\color{purple}n}({\color{olive}x},{\color{teal}s}) = \sum_{{\color{olive}k}=0}^{\lfloor ({\color{purple}n}-1)/2\rfloor} 
\binom{{\color{purple}n}-1-{\color{olive}k}}{{\color{olive}k}}  {\color{teal}s}^{\color{olive}k} {\color{olive}x}^{{\color{purple}n}-1-2{\color{olive}k}}.
\end{equation}
Their ordinary generating function is
\begin{equation}\label{Fibo:gf}
\sum_{{\color{purple}n}\geq 0} {\color{blue}F}_{\color{purple}n}({\color{olive}x},{\color{teal}s}) {\color{purple}Y}^{\color{purple}n} = \frac{{\color{purple}Y}}{1-{\color{olive}x}{\color{purple}Y} -{\color{teal}s}{\color{purple}Y}^2}.
\end{equation}
These polynomials obey the three-term recurrence
\begin{equation}\label{Fibo:rec}
{\color{blue}F}_{\color{purple}n}({\color{olive}x},{\color{teal}s}) = {\color{olive}x} {\color{blue}F}_{{\color{purple}n}-1}({\color{olive}x},{\color{teal}s}) + {\color{teal}s} {\color{blue}F}_{{\color{purple}n}-2}({\color{olive}x},{\color{teal}s}),
\end{equation}
subject to the initial values ${\color{blue}F}_0({\color{olive}x},{\color{teal}s})=0$ and ${\color{blue}F}_1({\color{olive}x},{\color{teal}s})=1$.

\begin{Lemma}
	The Fibonacci polynomials admit the representation
\begin{equation*}
{\color{blue}F}_{\color{purple}n}({\color{olive}x},{\color{teal}s}) =   \frac{1}{2^{{\color{purple}n}-1}} 
	\sum_{{\color{olive}k}=1}^{\lfloor ({\color{purple}n}+1)/2\rfloor} \binom{{\color{purple}n}}{2{\color{olive}k}-1} ({\color{olive}x}^2+4{\color{teal}s})^{{\color{olive}k}-1} {\color{olive}x}^{{\color{purple}n}-2{\color{olive}k}+1}.
\end{equation*}

\end{Lemma}

\begin{proof}
Let $\RHS$ denote the right-hand side of the preceding identity. Set ${\color{red}u} = {\color{olive}x}^2 + 4{\color{teal}s}$. Then we obtain
\begin{align*}
	\sum_{{\color{red}m}\geq 0} \RHS \  {\color{purple}Y}^{\color{purple}n} &=\sum_{{\color{purple}n}\geq 0}
  \frac{2}{(2)^{{\color{purple}n}}} 
	\sum_{{\color{olive}k}\geq 1, 2{\color{olive}k}-1\leq {\color{purple}n}} \binom{{\color{purple}n}}{2{\color{olive}k}-1} {\color{red}u}^{{\color{olive}k}-1}{\color{olive}x}^{{\color{purple}n}-2{\color{olive}k}+1}
	{\color{purple}Y}^{\color{purple}n}\\
&=\frac{2{\color{olive}x}}{{\color{red}u}}
\sum_{{\color{olive}k}\geq 1}
	\frac{  {\color{red}u}^{\color{olive}k}}{ {\color{olive}x}^{2{\color{olive}k}} }
	\sum_{{\color{purple}n}\geq 2{\color{olive}k}-1}
	 \binom{{\color{purple}n}}{2{\color{olive}k}-1} ( \frac{ {\color{olive}x} {\color{purple}Y}}{2})^{\color{purple}n}\\
&=\frac{2{\color{olive}x}}{{\color{red}u}}
\sum_{{\color{olive}k}\geq 1}
	\frac{  {\color{red}u}^{\color{olive}k}}{ {\color{olive}x}^{2{\color{olive}k}} }
		\frac{( \frac{ {\color{olive}x} {\color{purple}Y}}{2})^{2{\color{olive}k}-1}}
		{(1- \frac{ {\color{olive}x} {\color{purple}Y}}{2})^{2{\color{olive}k}}}\\
&=\frac{2{\color{olive}x}}{{\color{red}u}}
\frac{{\color{red}u} \frac{{\color{olive}x}{\color{purple}Y}}{2}}
		{{\color{olive}x}^2  (1-{\color{olive}x}{\color{purple}Y}/2)^2}
\sum_{{\color{olive}k}\geq 0}
	\left(	\frac{  {\color{red}u}}{ {\color{olive}x}^{2} }
		\frac{( \frac{ {\color{olive}x} {\color{purple}Y}}{2})^{2}}
		{(1- \frac{ {\color{olive}x} {\color{purple}Y}}{2})^{2}} \right)^{\color{olive}k}\\
&=\frac{  {\color{purple}Y}}
		{ (1-{\color{olive}x}{\color{purple}Y}/2)^2}
\frac{1}{1-
		\frac{  {\color{red}u}}{ {\color{olive}x}^{2} }
		\frac{( \frac{ {\color{olive}x} {\color{purple}Y}}{2})^{2}}
		{(1- \frac{ {\color{olive}x} {\color{purple}Y}}{2})^{2}} }\\
&=\frac{  {\color{purple}Y} }{   1 - {\color{olive}x}{\color{purple}Y}  - {\color{teal}s} {\color{purple}Y}^2}.\qedhere
\end{align*}

\end{proof}

The Lucas polynomials are defined in Section \ref{sec:notations}. Their generating function is given by
\begin{equation}\label{Lucas:gf}\sum_{{\color{purple}n}\geq 0} {\color{cyan}L}_{\color{purple}n}({\color{olive}x},{\color{teal}s}) {\color{purple}Y}^{\color{purple}n} = \frac{2-{\color{olive}x}{\color{purple}Y}}{1-{\color{olive}x}{\color{purple}Y} -{\color{teal}s}{\color{purple}Y}^2}.
\end{equation}

\begin{Lemma}
	The Lucas polynomials admit the explicit representation
	\begin{equation*}
{\color{cyan}L}_{\color{purple}n}({\color{olive}x},{\color{teal}s}) = \frac{1}{2^{{\color{purple}n}-1}} 
	\sum_{{\color{olive}k}=0}^{\lfloor {\color{purple}n}/2\rfloor} \binom{{\color{purple}n}}{2{\color{olive}k}} ({\color{olive}x}^2+4{\color{teal}s})^{\color{olive}k} {\color{olive}x}^{{\color{purple}n}-2{\color{olive}k}}.
\end{equation*}
\end{Lemma}

\begin{proof}
Let ${\color{red}u} = {\color{olive}x}^2 + 4{\color{teal}s}$, and denote the right-hand side of the identity above by $\RHS$. Then,
\begin{align*}
	\sum_{{\color{purple}n}\geq 0} \RHS {\color{purple}Y}^{\color{purple}n} &=\sum_{{\color{purple}n}\geq 0}
  \frac{2}{(2)^{{\color{purple}n}}} 
	\sum_{{\color{olive}k}\geq 0, 2{\color{olive}k}\leq {\color{purple}n}} \binom{{\color{purple}n}}{2{\color{olive}k}} {\color{red}u}^{\color{olive}k} {\color{olive}x}^{{\color{purple}n}-2{\color{olive}k}}
	{\color{purple}Y}^{\color{purple}n}\\
&=\sum_{{\color{olive}k}\geq 0}
	\frac{ 2 {\color{red}u}^{\color{olive}k}}{ {\color{olive}x}^{2{\color{olive}k}} }
	\sum_{{\color{purple}n}\geq 2{\color{olive}k}}
	 \binom{{\color{purple}n}}{2{\color{olive}k}} ( \frac{ {\color{olive}x} {\color{purple}Y}}{2})^{\color{purple}n}\\
&=\sum_{{\color{olive}k}\geq 0}
	\frac{ 2 {\color{red}u}^{\color{olive}k}}{ {\color{olive}x}^{2{\color{olive}k}} }
		\frac{( \frac{ {\color{olive}x} {\color{purple}Y}}{2})^{2{\color{olive}k}}}
		{(1- \frac{ {\color{olive}x} {\color{purple}Y}}{2})^{2{\color{olive}k}+1}}\\
&=\frac{2}
		{1- \frac{ {\color{olive}x} {\color{purple}Y}}{2}}
\sum_{{\color{olive}k}\geq 0}
	\left(	\frac{  {\color{red}u}}{ {\color{olive}x}^{2} }
		\frac{( \frac{ {\color{olive}x} {\color{purple}Y}}{2})^{2}}
		{(1- \frac{ {\color{olive}x} {\color{purple}Y}}{2})^{2}} \right)^{\color{olive}k}\\
&=\frac{2} {1- \frac{ {\color{olive}x} {\color{purple}Y}}{2}}
	\frac{1} {1- \frac{ {\color{red}u}{\color{purple}Y}^2}{(2-{\color{olive}x}{\color{purple}Y})^2}}\\
&=\frac{2-{\color{olive}x}{\color{purple}Y}}  {1- {\color{olive}x} {\color{purple}Y} -{\color{teal}s}{\color{purple}Y}^2}.\qedhere
\end{align*}
\end{proof}

We had examined the odd case, but without success. Instead, we choose to present the quadratic equation here, in case the reader wishes to continue working on the problem. Since the functions ${\color{blue}\gamma}^{(2{\color{olive}\tau}+1)}_{\color{purple}n}({\color{cyan}t})$ are defined by
$$\sum {\color{blue}\gamma}^{(2{\color{olive}\tau}+1)}_{\color{purple}n}({\color{cyan}t})\, {\color{magenta}q}^{\color{purple}n} \;=\; {\color{blue}\gamma}({\color{cyan}t},{\color{magenta}q})\,{\color{red}G}({\color{cyan}t},{\color{magenta}q})^{\color{olive}\tau},$$
it appears natural to examine the series
$$
{\color{blue}F} = \frac{{\color{blue}\gamma}({\color{cyan}t},{\color{magenta}q})\,({\color{blue}\gamma}({\color{cyan}t},{\color{magenta}q})-1)^{\color{red}m}}{{\color{magenta}q}^{{\color{red}m}_0}}.
$$
This led us to the following lemma, whose proof we omit here.

\begin{Lemma}
The power series ${\color{blue}F}$ defined above satisfies the quadratic relation
$$
0 = -2{\color{magenta}q}^{{\color{red}m}-{\color{red}m}_0} + {(-1)^{\color{red}m} {\color{blue}T}}\,{\color{blue}F} - 2{\color{cyan}t}^{{\color{red}m}+1} {\color{magenta}q}^{{\color{red}m}+{\color{red}m}_0+1} {\color{blue}F}^2,
$$
where
\begin{align*}
  {\color{blue}T} &= (1-{\color{magenta}q}+{\color{cyan}t}{\color{magenta}q})\, {\color{cyan}L}_{\color{purple}n}(-1+{\color{magenta}q}+{\color{cyan}t}{\color{magenta}q} , -{\color{cyan}t}{\color{magenta}q}^2) \\
    &\qquad - \Bigl((1+{\color{magenta}q}-{\color{cyan}t}{\color{magenta}q})^2 - 4{\color{magenta}q}\Bigr)\, {\color{blue}F}_{\color{purple}n}(-1+{\color{magenta}q}+{\color{cyan}t}{\color{magenta}q} , -{\color{cyan}t}{\color{magenta}q}^2).
\end{align*}
\end{Lemma}

It is interesting to observe that this quadratic equation simultaneously involves both Lucas and Fibonacci polynomials.

\bibliographystyle{plain}


\begin{thebibliography}{10}

\bibitem{Aigner1999_Catalanlike}
Martin Aigner.
\newblock Catalan-like numbers and determinants.
\newblock {\em J. Combin. Theory Ser. A}, 87(1):33--51, 1999.

\bibitem{Chamberland2007French_GenCatalan}
Marc Chamberland and Christopher French.
\newblock Generalized {C}atalan numbers and generalized {H}ankel
  transformations.
\newblock {\em J. Integer Seq.}, 10(1):Article 07.1.1, 7, 2007.

\bibitem{Chang2013HuZhang_Hdet}
Xiang-Ke Chang, Xing-Biao Hu, and Ying-Nan Zhang.
\newblock A direct method for evaluating some nice {H}ankel determinants and
  proofs of several conjectures.
\newblock {\em Linear Algebra Appl.}, 438(5):2523--2541, 2013.

\bibitem{Cigler2021Kratt_linear}
J.~Cigler and C.~Krattenthaler.
\newblock Hankel determinants of linear combinations of moments of orthogonal
  polynomials.
\newblock {\em Int. J. Number Theory}, 17(2):341--369, 2021.

\bibitem{Cigler2011}
Johann Cigler.
\newblock {S}ome nice {H}ankel determinants.
\newblock {\em arXiv}, 1109.1449, 2011.

\bibitem{Cigler2013spe}
Johann Cigler.
\newblock {A} special class of {H}ankel determinants.
\newblock {\em arXiv}, 1302.4235, 2013.

\bibitem{Cigler2018}
Johann Cigler.
\newblock Catalan numbers, {H}ankel determinants and {F}ibonacci polynomials.
\newblock {\em arXiv}, 1801.05608, 2018.

\bibitem{Cigler2019}
Johann Cigler.
\newblock Some observations about determinants which are connected with
  {C}atalan numbers and related topics.
\newblock {\em arXiv}, 1902.10468, 2019.

\bibitem{Cigler2023_Experimental}
Johann Cigler.
\newblock {S}ome experimental observations about {H}ankel determinants of
  convolution powers of {C}atalan numbers.
\newblock {\em arXiv}, 2308.07642, 2023.

\bibitem{Cigler2024ConvCatalan}
Johann Cigler.
\newblock {H}ankel determinants of convolution powers of {C}atalan numbers
  revisited.
\newblock {\em arXiv}, 2403.11244, 2024.

\bibitem{Cvetkovic2002RI_CatalanHDet}
Aleksandar Cvetkovi\'c, Predrag Rajkovi\'c, and Milo\v~s Ivkovi\'c.
\newblock Catalan numbers, and {H}ankel transform, and {F}ibonacci numbers.
\newblock {\em J. Integer Seq.}, 5(1):Article 02.1.3, 8, 2002.

\bibitem{DeSainte1986Viennot}
Myriam de~Sainte-Catherine and G\'erard Viennot.
\newblock Enumeration of certain {Y}oung tableaux with bounded height.
\newblock In {\em Combinatoire \'enum\'erative ({M}ontreal, {Q}ue.,
  1985/{Q}uebec, {Q}ue., 1985)}, volume 1234 of {\em Lecture Notes in Math.},
  pages 58--67. Springer, Berlin, 1986.

\bibitem{Egecioglu2009_CatalanHdet}
\"Omer E\u{g}ecio\u{g}lu.
\newblock A {C}atalan-{H}ankel determinant evaluation.
\newblock In {\em Proceedings of the {F}ortieth {S}outheastern {I}nternational
  {C}onference on {C}ombinatorics, {G}raph {T}heory and {C}omputing}, volume
  195, pages 49--63, 2009.

\bibitem{Egecioglu2007RR_AlmostHDet}
\"Omer E\u{g}ecio\u{g}lu, Timothy Redmond, and Charles Ryavec.
\newblock Almost product evaluation of {H}ankel determinants.
\newblock {\em Electron. J. Combin.}, 15(1):Research Paper 6, 58, 2008.

\bibitem{Gessel2008slide}
Ira~M. Gessel.
\newblock {T}he {M}ethode of {C}oefficients.
\newblock {\em slide available on the author's personal homepage}, 2008.

\bibitem{Gessel2006Xin}
Ira~M. Gessel and Guoce Xin.
\newblock The generating function of ternary trees and continued fractions.
\newblock {\em Electron. J. Combin.}, 13(1):Research Paper 53, 48, 2006.

\bibitem{Han2016Adv}
Guo-Niu Han.
\newblock Hankel continued fraction and its applications.
\newblock {\em Adv. Math.}, 303:295--321, 2016.

\bibitem{Han2020Trans}
Guo-Niu Han.
\newblock Hankel continued fractions and {H}ankel determinants of the {E}uler
  numbers.
\newblock {\em Trans. Amer. Math. Soc.}, 373(6):4255--4283, 2020.

\bibitem{Han2025Pedon}
Guo-Niu Han and Emmanuel Pedon.
\newblock Hankel continued fractions and hankel determinants for ${\color{magenta}q}$-deformed
  metallic numbers.
\newblock {\em arxiv}, 2502.05993, 2025.

\bibitem{Krattenthaler1998}
C.~Krattenthaler.
\newblock Advanced determinant calculus.
\newblock {\em S\'em. Lothar. Combin.}, 42:Art. B42q, 67 pp., 1999.
\newblock {\tt http://www.mat.univie.ac.at/\~{}slc/}.

\bibitem{Krattenthaler2005}
C.~Krattenthaler.
\newblock Advanced determinant calculus: a complement.
\newblock {\em Linear Algebra Appl.}, 411:68--166, 2005.

\bibitem{Krattenthaler2010_genCatalan}
C.~Krattenthaler.
\newblock Determinants of (generalised) {C}atalan numbers.
\newblock {\em J. Statist. Plann. Inference}, 140(8):2260--2270, 2010.

\bibitem{Krattenthaler2021_linearII}
C.~Krattenthaler.
\newblock Hankel determinants of linear combinations of moments of orthogonal
  polynomials, {II}.
\newblock {\em Ramanujan J.}, 61(2):597--627, 2023.

\bibitem{Rajkovic2007PB_HCatalan}
Predrag~M. Rajkovi\'c, Marko~D. Petkovi\'c, and Paul Barry.
\newblock The {H}ankel transform of the sum of consecutive generalized
  {C}atalan numbers.
\newblock {\em Integral Transforms Spec. Funct.}, 18(3-4):285--296, 2007.

\bibitem{Tamm2001_Aspects}
Ulrich Tamm.
\newblock Some aspects of {H}ankel matrices in coding theory and combinatorics.
\newblock {\em Electron. J. Combin.}, 8(1):Article 1, 31, 2001.

\bibitem{Wang2018Xin_Convolution}
Ying Wang and Guoce Xin.
\newblock Hankel determinants for convolution powers of {C}atalan numbers.
\newblock {\em Discrete Math.}, 342(9):2694--2716, 2019.

\end{thebibliography}

\end{document}